\documentclass[12pt,reqno,a4paper]{amsart}

\usepackage{mdframed}
\usepackage{setspace}
\usepackage{mathrsfs}

\usepackage[latin1]{inputenc}

\usepackage[linkcolor=red,colorlinks=true]{hyperref}

\usepackage{etoolbox}
\usepackage{amsmath}
\usepackage{enumerate}
\usepackage{amssymb}
\usepackage{amscd}
\usepackage{amsthm}
\usepackage{amsfonts}
\usepackage{graphicx}
\usepackage[all,cmtip]{xy}
\usepackage{enumitem}

\patchcmd{\subsection}{-.5em}{.5em}{}{}
\patchcmd{\subsubsection}{-.5em}{.5em}{}{}

\usepackage{enumitem}

\makeatother
\usepackage{hyperref}

\numberwithin{equation}{section}

\newcommand{\SL}{\operatorname{SL}}

\newcommand{\cL}{\mathcal{L}}

\newcommand{\cN}{\mathcal{N}}

\newcommand{\bN}{\mathbb{N}}

\newcommand{\bR}{\mathbb{R}}

\newcommand{\bZ}{\mathbb{Z}}

\newcommand{\ra}{\rightarrow}

\newcommand{\qand}{\quad \textrm{and} \quad}

\newcommand\subsetsim{\mathrel{%
\ooalign{\raise0.2ex\hbox{$\subset$}\cr\hidewidth\raise-0.8ex\hbox{\scalebox{0.9}{$\sim$}}\hidewidth\cr}}}
\newcommand{\eps}{\varepsilon}

\DeclareMathOperator{\Ad}{Ad}

\DeclareMathOperator{\supp}{supp}

\definecolor{lichtgrijs}{gray}{0.95}
\mdfdefinestyle{mystyle}{ %
    backgroundcolor=lichtgrijs, %
    linewidth=0pt, %
    innertopmargin=10pt, %
    innerbottommargin=10pt,%
    nobreak=true
}

\newtheorem{theorem}{Theorem}[section]
\newtheorem{corollary}[theorem]{Corollary}
\newtheorem{proposition}[theorem]{Proposition}
\newtheorem{lemma}[theorem]{Lemma}

\theoremstyle{definition}

\newtheorem{remark}[theorem]{Remark}

\newtheorem{example}{Example}[section]

\usepackage{changepage}
\usepackage{mathtools}
\usepackage{graphicx}
\usepackage{calc}

\usepackage[toc,page]{appendix}

\usepackage{scalerel}

\usepackage{extarrows}

\begin{document}
\bibliographystyle{plain} 

\title[Decorrelation estimates for translated measures]{Decorrelation estimates for translated measures under diagonal flows}

\author{Michael Bj\"orklund}

\author{Reynold Fregoli} 

\author{Alexander Gorodnik}

\maketitle

\begin{abstract}
A profound link between Homogeneous Dynamics and Diophantine Approximation is based on 
an observation that Diophantine properties of a real matrix $B$ are encoded by 
the corresponding lattice $\Lambda_B$ translated by a multi-parameter semigroup $a(t)$.
We establish quantitative decorrelation estimates for measures supported on leaves $a(t)\Lambda_B$ with the error terms depending only on the minimum
of the pairwise distances between the parameters. The proof involves a careful
analysis of the translated measures 
in the products of the spaces of unimodular lattices
and establishes quantitative equidistributions  to 
measures supported on various intermediate homogeneous subspaces.
\end{abstract}

\setcounter{tocdepth}{1}
\tableofcontents

\section{Introduction}

When studying Diophantine properties of a real $(m\times n)$-matrix 
$B$, it is often natural to consider the corresponding unimodular lattice
$$
\Lambda_B:=\{ (p+Bq,q):\,\, p\in \bZ^m,\, q\in \bZ^n\} \subset  \bR^{m+n},
$$
and view it as a point in the space $X$ of all $(m+n)$-dimensional unimodular 
lattices. It turns out that many Diophantine properties of $B$ can be equivalently expressed in
terms of the asymptotic behaviour of the orbit of $\Lambda_B$ under the \emph{diagonal flow}
$$
a(t):=\hbox{diag}\big(e^{t^{(1)}},\ldots, e^{t^{(m)}},e^{-t^{(m+1)}},\ldots, e^{-t^{(m+n)}}\big) \in \SL_{m+n}(\bR),
$$
where $t=(t^{(1)},\ldots,t^{(m+n)})$ belongs to the additive sub-semigroup $A^{+} \subset \bR_{+}^{m+n}$
defined by
$$
A^+:=\big\{t \in \bR_+^{m+n}:\, t^{(1)}+\cdots+t^{(m)}=t^{(m+1)}+\cdots+t^{(m+n)}\big\}.
$$
For instance, Kleinbock and Margulis \cite{KM0} observed
that Khinchin's classical theorem in Diophantine approximate is intimately 
related to the visits to shrinking cuspidal neighborhoods in $X$ of the orbit
of $\Lambda_B$ under the sub-flow $g_t = a(t/m,\ldots,t/m,t/n,\ldots,t/n)$ for $t \geq 0$.
In particular, the crucial Borel-Cantelli property needed in their proof was derived
from the asymptotic "independence" of the lattices $g_{t_1} \Lambda_B$ and $g_{t_2} \Lambda_B$
as $t_1, t_2$ and $|t_1-t_2|$ all tend to infinity. More generally, the orbit of $\Lambda_B$ under larger sub-flows of $a(A^{+})$ encode 
\emph{multiplicative} Diophantine approximation properties of the matrix $B$ (see e.g. \cite{BFG1}). \\

Systematic implementation of the idea that various objects in Number Theory
are asymptotically independent goes back to the monograph \cite{Ph}
of W.~Philipp, who developed a general framework for proving probabilistic limit theorems
for weakly independent processes having arithmetic applications in mind.
In particular, this led to profound limit theorems for continued fractions,
where the asymptotic independence was derived from properties of the Gauss map.
In the higher-dimensional setting, the asymptotic independence property of translated lattices $a(t)\Lambda_B$
provides a promising approach to generalize these ideas
and hopefully to establish new limit theorems in Diophantine Approximation
and Geometry of Numbers. For instance, we refer to recent works \cite{AG,BG1,BG3,BG4,DFV,DFL},
where this approach was successfully applied. \\


Our paper is devoted to proving a general asymptotic independence result for lattices of the form $a(t)\Lambda_B$ with a random $(m \times n)$-matrix $B$.
To make this more precise, let us introduce some notation. Fix integers $m,n \geq 1$ once and for all. Let $X$ denote the space of all unimodular lattices in $\bR^{m+n}$ and let $\mu$ denote the unique $\SL_{m+n}(\bR)$-invariant probability measure on $X$. Define the sub-manifold  
$$
Y:=\big\{\Lambda_B:\, B\in \hbox{Mat}_{m\times n}(\bR/\bZ)\big\}\subset X,
$$
which can be identified with the standard $(m+n)$-torus $\bR^{m+n}/\bZ^{m+n}$, and equip $Y$ with the Lebesgue measure $\nu$ coming from this torus. Our aim is to show that, given any smooth 
functions $\varphi_1,\ldots,\varphi_r$ on $X$, and $t_1,\ldots, t_r \in A^{+}$, we have
\begin{equation}
\label{approx}
\int_Y \left(\prod_{s=1}^r \varphi_s\circ a(t_s) \right)\, d\nu
\approx
\prod_{s=1}^r \int_Y \varphi_s\circ a(t_s) \, d\nu
\end{equation}
with an explicit error term which \emph{only} depends on the function
$$
\Delta(t_1,\ldots, t_r):=\min_{i \ne j} \|t_i-t_j\|,
$$
where $\|\cdot\|$ denotes the maximum norm on $\bR^{m+n}$. Note that \eqref{approx} only involves
correlations with respect to $\nu$, and not $\mu$. We refer to \eqref{approx} as (higher order) \emph{decorrelation}, as opposed to (higher order) \emph{equidistribution} when one instead proves \eqref{approx} with $\mu$ on the right-hand side, but often with a significantly worse error term (see e.g. \cite{BFG1, BG1,BG2}), which is really only useful when all coordinates in $t$ tend to infinity. 
More precisely, the error term in this case depend on the minimum of $\Delta(t_1,\ldots,t_r)$
and the function
\[
\min(\lfloor t_1 \rfloor,\ldots, \lfloor t_r \rfloor),
\]
where $\lfloor t \rfloor = \min t^{(s)}$ for $t = (t^{(1)},\ldots, t^{(r)}) \in A^{+}$. To prove \eqref{approx} we need to (quantitatively) understand \emph{all} limits of translated measures $a(t)\nu$ as $t$ tends to infinity in $A^{+}$ (allowing for some coordinates in $t$ to stay bounded). \\

The problem of limiting distribution of measures
$a(t)_*\nu$ in $X$  
can be studied using Ratner`s measure classification \cite{R} for measures invariant under unipotent flows.
In particular, the case of measures supported on horospherical orbits was handled by Dani \cite{D1,D2} and generalized  
by Shah \cite{Shah}. 
More generally, this approach can be potentially used to study
the distribution of translated measures embedded diagonally in $X\times \cdots \times X$, but
this method does not provide quantitative bounds.
Regarding quantitative estimates,
when $a_t$ is a \emph{one-parameter} partially hyperbolic flow and 
$\nu$ is a smooth measure supported on an unstable leaf of $a_t$, the quantitative asymptotic behaviour of the measures $(a_t)_*\nu$ as $t\to\infty$ 
was investigated by Kleinbock and Margulis in \cite{KM1}.
The main difficulty of our present setting is that 
$Y$ is not an unstable leaf for the family of transformations $a(t)$. 
The first quantitative result in this setting was established by Kleinbock and Margulis in \cite{KM2}:

\begin{theorem}[Kleinbock, Margulis]\label{th:km}
There exist $\ell\ge 1$ and $\delta>0$ such that for every $f\in C^\infty (Y)$, 
$\varphi\in C_c^\infty (X),$ and $t\in A^+$,
$$
\int_Y f\cdot(\varphi\circ a(t))\, d\nu=\left(\int_Y f\, d\nu\right)  	
\left(\int_X \varphi\, d\mu\right)+O\Big( e^{-\delta \lfloor t\rfloor}\, \|f\|_{C^\ell}\, \cN_\ell(\varphi)\Big),
$$
where $\cN_\ell(\varphi):=\max \left(\|\varphi\|_{C^0},\|\varphi\|_{Lip}, \|\varphi\|_{L^{2}_\ell}\right)$. Here $\| \cdot \|_{Lip}$ and $\|\cdot\|_{L^2_\ell}$ refers to the Lipschitz norm and
the $L^2$-Sobolev norm of order $\ell$ respectively.
\end{theorem}

Shi \cite{Shi} established the following extension to higher-order correlations:

\begin{theorem}[Shi]\label{th:shi}
	There exist $\ell\ge 1$ and  $\delta^\prime>0$ such that for every $f\in C^\infty (Y)$, 
	$\varphi_1,\ldots,\varphi_r\in C_c^\infty (X),$ and $t_1,\ldots,t_r\in A^+$,
	\begin{align*}
	\int_Y f\left(\prod_{s=1}^r \varphi_s\circ a(t_s)\right)\, d\nu=&\int_Y f\, d\nu
	\prod_{s=1}^r \int_X \varphi_s\, d\mu\\
	&\quad+O\left( e^{-\delta^\prime \min\big(\lfloor t_1\rfloor, \lfloor t_2-t_1\rfloor,\ldots,\lfloor t_r-t_{r-1}\rfloor\big)} \|f\|_{C^\ell}\prod_{s=1}^ r \cN_\ell(\varphi_s)\right).
	\end{align*}
\end{theorem}

Later, Bj\"orklund and Gorodnik \cite{BG2} established a similar extension but with a different error term: 

\begin{theorem}[Bj\"orklund, Gorodnik]\label{th:bg}
	There exist $d_r\ge 1$ and $\delta_r>0$ such that for every $f\in C^\infty (Y)$, 
	$\varphi_1,\ldots,\varphi_r\in C_c^\infty (X),$ and $t_1,\ldots,t_r\in A^+$,
	\begin{align*}
		\int_Y f\left(\prod_{s=1}^r \varphi_s\circ a(t_s)\right)\, d\nu=&\int_Y f\, d\nu
		\prod_{s=1}^r \int_X \varphi_s\, d\mu\\
		&\quad +O\left( e^{-\delta_r \min\big(\lfloor t_1\rfloor, \ldots,\lfloor t_r\rfloor, \Delta(t_1,\ldots,t_r)\big)} \|f\|_W\prod_{s=1}^ r \cN_{\ell, d_r}(\varphi_s)\right),
	\end{align*}
where 
$\cN_{\ell,d}(\varphi):=\max \left(\|\varphi\|_{C^0},\|\varphi\|_{Lip}, \|\varphi\|_{L^{2^d}_\ell}\right)$. Here $\|\cdot\|_W$ refers to the Wiener norm of functions on $Y$, defined by \eqref{eq:wiener_0} below.
\end{theorem}

The main result of this paper is:

\begin{theorem}\label{th:new}
	There exist $\ell\ge 1$ and $\eta_r>0$ such that for every 
	$\varphi_1,\ldots,\varphi_r\in C_c^\infty (X),$ and $t_1,\ldots,t_r\in A^+$,
	\begin{align*}
		\int_Y \left(\prod_{s=1}^r \varphi_s\circ a(t_s)\right)\, d\nu=
		\prod_{s=1}^r \int_Y \varphi_s\circ a(t_s)\, d\nu
		+O\left( e^{-\eta_r \Delta(t_1,\ldots,t_r)} \prod_{s=1}^r \|\varphi_s\|_{C^\ell}\right).
	\end{align*}
\end{theorem}

We again stress that the novelty here is that the error term depends only on the distances 
$\|t_i-t_j\|$ with $i\ne j$, while the previous results are only applicable when 
$\min\big(\lfloor t_1\rfloor, \ldots,\lfloor t_r\rfloor\big)\to \infty$.
We also note that the measures $a(t)_*\nu$, as $t \ra \infty$ in $A^{+}$, can have many 
different limits that need to be analysed in order to prove Theorem \ref{th:new}.
We illustrate this by the following example.

\begin{example}
Let $m=2$ and $n=1$. 
According to Theorem \ref{th:km}, for $\varphi\in C_c(X)$,
$$
\int_Y \varphi\circ a(s,t)\, d\nu \to \int_{X}\varphi\, d\mu
\quad\hbox{as $\min(s,t)\to\infty$.}
$$
However, the family of measures $a(s,t)_*\nu$ also exhibit 
other limiting behaviors: one can show, in particular, that
$$
\int_Y \varphi\circ a(s,0)\, d\nu \to \int_{X_1}\varphi\, d\mu_1\quad\hbox{as $s\to\infty$},
$$
and
$$
\int_Y \varphi\circ a(0,t)\, d\nu \to \int_{X_2}\varphi\, d\mu_2\quad\hbox{as $t\to\infty$},
$$
 where
$$
X_1 :=\left(
\begin{tabular}{ccc}
$*$ & $*$ & $*$ \\	
0 & $*$ & 0 \\	
$*$ & $*$ & $*$
\end{tabular}
\right)\Gamma
\quad\quad\hbox{and}\quad\quad
X_2 :=\left(
\begin{tabular}{ccc}
	$*$ & $*$ & $*$ \\	
	0 & $*$ & $*$ \\	
	0 & $*$ & $*$
\end{tabular}
	\right)\Gamma
$$
are the homogeneous subspaces
of $X$ equipped with invariant probability measures $\mu_1$ and $\mu_2$.
More generally, let us embed
$\nu$ in $X\times X$ diagonally and 
consider 
the family of measures $\big(a(s_1,t_1),a(s_2,t_2\big))_*\nu$ 
with $(s_1,t_1),(s_2,t_2)\in A^+$. Here one has 16 possible 
limiting behaviors. For example, one can show that 
for $\varphi_1,\varphi_2\in C_c(X)$,
\begin{equation}
	\label{eq:ex}
\int_Y \big(\varphi_1\circ a(s,0)\big)\cdot  \big(\varphi_2\circ a(0,t)\big)\, d\nu \to \left(\int_{X_1}\varphi_1\, d\mu_1\right) \left(\int_{X_2}\varphi_2\, d\mu_2\right)\quad \hbox{as $s,t\to\infty$.}
\end{equation}
In order to prove Theorem \ref{th:new},
one needs to analyze all such possibilities.
Ultimately, a general version of \eqref{eq:ex},
which describes various limiting behaviors,
is given by Theorem~\ref{th:multi_gen} below.
\end{example}

\noindent \textbf{Organization of the paper.}
In Section \ref{sec:notation}, we introduce basic notation and, in particular, a family of homogeneous subspaces $X_I$ that will play a crucial
role for the rest of the paper.
In Section \ref{sec:mixing}, we study mixing and equidistribution properties on those spaces.
While for certain one-parameter subgroups a required de-correlation
estimate is given by Corollary \ref{cor:decor}, 
this fails for other one-parameter subgroups. 
This issue in addressed in Section \ref{sec:decor},
where additional decorrelation estimates are established.
 In Section \ref{sec:weight}, we prove an  
elementary result about weights that subsequently will allow us to
run our inductive argument using the decorrelation estimates from  
Sections \ref{sec:mixing}--\ref{sec:decor}.
In Section \ref{sec:multi}, we prove a general multiple equidistribution
result (Theorem \ref{th:multi_gen}) generalizing Theorem \ref{th:bg}.
Finally, Theorem \ref{th:new} is proved in Sections \ref{sec:separated}--\ref{sec:proof} with a help of Theorem \ref{th:multi_gen}.

\section{Basic notation} \label{sec:notation}

Throughout the paper, $X$ denotes the space of unimodular lattice in $\bR^{m+n}$:
$$
X\simeq G/\Gamma,\quad \hbox{with $G:=\hbox{SL}_{m+n}(\bR)$ and $\Gamma:=\hbox{SL}_{m+n}(\bZ)$},
$$
equipped with the unique $G$-invariant probability measure $\mu$.
We consider the orbit
$$
Y:=U\Gamma\subset X
$$
of the group
$$
U:=\left(
\begin{tabular}{cc}
	$I_m$ & $*$ \\
	$0$ & $I_n$
\end{tabular}\right),
$$
which can be identified with $(m\times n)$-dimensional torus.
We equip $Y$ with the normalized Lebesgue measure $\nu$ (which is the unique $U$-invariant probability measure on the $U$-orbit $Y$).


\medskip

Let us now introduce a family of intermediate homogeneous subspaces $X_I$
of $X$ that will play a crucial role in our arguments.
We call a subset $I\subset \{1,\ldots, m+n\}$ \emph{admissible} if
\begin{equation}\label{eq:I}
I\cap \{1,\ldots, m\}\ne \emptyset \quad \hbox{and} \quad 
I\cap \{m+1,\ldots, m+n\}\ne \emptyset.
\end{equation}
Let us first suppose for simplicity that
\begin{equation}\label{eq:i_0}
I= I_0:=\{m-k_1+1,\ldots,m+k_2\}
\end{equation}
for some $k_1,k_2\ge 1$. In this case, we define
$$
G_{I_0}:=
\left(
\begin{tabular}{ccc}
$I_{m-k_1}$ & $*$ & $*$ \\
$0$ & $\hbox{SL}_{k_1+k_2}(\bR)$ & $*$ \\
$0$ & $0$ & $I_{n-k_1}$
\end{tabular}
\right).
$$
We note that $G_{I_0}= S_{I_0}\ltimes U_{I_0}$, where
$$
S_{I_0}:=
\left(
\begin{tabular}{ccc}
	$I_{m-k_1}$ & $0$ & $0$ \\
	$0$ & $\hbox{SL}_{k_1+k_2}(\bR)$ & $0$ \\
	$0$ & $0$ & $I_{m-k_1}$
\end{tabular}
\right)
\quad\hbox{and}\quad
U_{I_0}=
\left(
\begin{tabular}{ccc}
	$I_{m-k_1}$ & $*$ & $*$ \\
	$0$ & $I_{k_1+k_2}$ & $*$ \\
	$0$ & $0$ & $I_{m-k_1}$
\end{tabular}
\right).
$$
For general admissible $I$, we take a permutation $\sigma_I$ such that 
\begin{align*}
\sigma_I (\{1,\ldots, m\}\backslash I) &=\{1, \ldots, m-k_1\},\\
\sigma_I (I\cap \{1,\ldots, m\}) &=\{m-k_1+1,\ldots,m\}, \\
\sigma_I(I\cap \{m+1,\ldots, m+n\}) &=\{m+1,\ldots,m+k_2\},\\
\sigma_I(\{m+1,\ldots, m+n\}\backslash I) &=\{m+k_2+1,\ldots,m+n\}.
\end{align*}
Let $w_I\in G$ be the corresponding permutation matrix and define
$$
G_I:= w_I G_{I_0} w_I^{-1}, \quad S_I:= w_I S_{I_0} w_I^{-1},\quad U_I:= w_I U_{I_0} w_I^{-1},
$$
and $G_I=S_I\ltimes U_I$. The orbits 
$$
X_I := G_I \Gamma \subset X
$$	
are finite volume homogeneous spaces. We denote by $\mu_I$ the unique $G_I$-invariant probability 
measure on $X_I$. For instance, when $I=\{1,\ldots,m+n\}$, we have $X_I=X$.


\medskip

We recall that the space
$$
Y=U\Gamma\simeq \bR^{mn}/\bZ^{mn}
$$
is (canonically) isomorphic to a torus, and under this isomorphism the measure $\nu$ corresponds to the Lebesgue probability measure. Let $\Xi_Y$ denote the set of continuous characters of $Y$ (viewed as a torus).
For $f\in L^1(Y)$, we write
$$
\widehat f(\xi):= \int_Y f\cdot \bar \xi\, d\nu.
$$
The \emph{Wiener norm} of $f$ is defined by
\begin{equation}\label{eq:wiener_0}
\|f\|_W:=\sum_{\xi\in\Xi_Y} |\widehat f(\xi)|.
\end{equation}
When $\|f\|_W<\infty$, one has a uniformly convergent expansion
\begin{equation}\label{eq:wiener}
f=\sum_{\xi\in\Xi_Y} \widehat f(\xi) \xi.
\end{equation}
We recall that for $f\in C^\infty(Y)$,
\begin{equation}
	\label{eq:wiener_n}
\|f\|_{C^0}\le \|f\|_{W}\le \|f\|_{C^1},
\end{equation}
and for $f_1,f_2\in C^\infty(Y)$,
$$
\|f\|_{W}\le \|f_1\|_{W}\cdot \|f_2\|_{W}.
$$
In particular, it follows that for $\phi\in C_c^\infty(X)$,
$$
\big\|\phi|_Y\big\|_W\le \|\phi\|_{C^1}.
$$



\section{Mixing and equidistribution in intermediate homogeneous spaces}
\label{sec:mixing}

Let $I\subset \{1,\ldots, m+n\}$ be an admissible subset. If $s \in G$, we write $\|s\|$ the operator norm of $s$ with respect to the $\ell^\infty$-norm of $\bR^{m+n}$.

\begin{theorem}\label{th:mixing_I}
The action of $S_I$ on $X_I$ is exponentially mixing,
that is, 
there exist $\ell\ge 1$ and $\tau>0$ such that for every 
$\varphi,\psi\in C_c^\infty (X_I)$ and $s\in S_I$,
$$
\int_{X_I} (\varphi\circ s)\cdot \psi\, d\mu_I=\left(\int_{X_I} \varphi\, d\mu_I\right)  	
\left(\int_{X_I} \psi\, d\mu_I\right)+O\Big( \max\big(1,\|s\|\big)^{-\tau}\, \|\varphi\|_{L^2_\ell} \|\psi\|_{L^2_\ell}\Big).
$$
\end{theorem}

\begin{proof}
Without loss of generality, we may assume that $I$ is given as in \eqref{eq:i_0}. In particular, $S_I\simeq \hbox{SL}_{k_1+k_2}(\bR)$.
The exponential mixing property will follow once we show that 
the action of $S_I$ on $X_I$ has the spectral gap property,
namely, that the unitary representation of $S_I$ on $L_0^2(X_I)$,
the space of $L^2$-integrable functions with integral zero,
is isolated from the trivial representation (see, for instance, \cite[Sec.~2.4]{KM1}).
	
We observe that the space of $U_I$-invariant functions 
$L^2(X_I)^{U_I}$ is invariant under $G_I$ and 
$$
L_0^2(X_I)^{U_I}\simeq L_0^2\big(\hbox{SL}_{k_1+k_2}(\bR)/\hbox{SL}_{k_1+k_2}(\bZ)\big)
$$
as $S_I$-spaces. The spectral gap property on the latter space is well-known.
Therefore, it is sufficient to verify the spectral gap property on the orthogonal complement $\mathcal{H}_I$ of $L_0^2(X_I)^{U_I}$.
We observe that 
$$
U_I=U^{(1)}_I U^{(2)}_I U^{(3)}_I,
$$
where 
\begin{align*}
	U^{(1)}_I &:=
	\left(
	\begin{tabular}{ccc}
		$I_{m-k_1}$ & $*$ & $0$ \\
		$0$ & $I_{k_1+k_2}$ & $0$ \\
		$0$ & $0$ & $I_{n-k_2}$
	\end{tabular}
	\right),\\
	U^{(2)}_I &:=
	\left(
	\begin{tabular}{ccc}
		$I_{m-k_1}$ & $0$ & $0$ \\
		$0$ & $I_{k_1+k_2}$ & $*$ \\
		$0$ & $0$ & $I_{n-k_2}$
	\end{tabular}
	\right),\\
	U^{(3)}_I &:=
	\left(
	\begin{tabular}{ccc}
		$I_{m-k_1}$ & $0$ & $*$ \\
		$0$ & $I_{k_1+k_2}$ & $0$ \\
		$0$ & $0$ & $I_{n-k_2}$
	\end{tabular}
	\right).
\end{align*}
We note that $U_I$ is generated by $U^{(1)}_I$ and $U^{(2)}_I$.

Let $V_i$, $i=1\ldots, m-k_1$, denote the subgroups $U^{(1)}_I$
corresponding to the $i$-th row in $\hbox{M}_{m-k_1,k_1+k_2}(\bR)$
and $W_j$, $j=m+k_2+1,m+n$, the subgroup of  
$U^{(2)}_I$
corresponding to the $j$-th column in $\hbox{M}_{k_1+k_2,n-k_2}(\bR)$.
We note that 
$$
S_IV_i\simeq \hbox{SL}_{k_1+k_2}(\bR)\ltimes \bR^{k_1+k_2}\quad\hbox{and}\quad
S_IW_j\simeq \hbox{SL}_{k_1+k_2}(\bR)\ltimes \bR^{k_1+k_2}.
$$

When $k_1+k_2\ge 3$, the group $S_I\simeq \hbox{SL}_{k_1+k_2}(\bR)$
has Kazhdan property (T), so it is sufficient to check that 
$\mathcal{H}_I$ contains no non-zero $S_I$-invariant vectors.
Suppose that such a vector exists. 
We apply Mautner Lemma (see, for instance, \cite[Lemma~1.4.8]{BHV}) to the representations of 
$S_IV_i$ and $S_IW_j$, and deduce that this vector is also invariant under
all $V_i$'s and $W_j$'s. 
Hence, it is invariant under $U^{(1)}_I$ and $U^{(2)}_I$,
so that it is $U_I$-invariant and hence is zero
because $\mathcal{H}_I$ is the orthogonal complement of $L_0^2(X_I)^{U_I}$.
This verifies that $\mathcal{H}_I$ contains no non-zero $S_I$-invariant vectors and implies the spectral gap property.

Now suppose that $k_1+k_2=2$, that is, $S_I\simeq \hbox{SL}_2(\bR)$.
In this case, we use that 
for any unitary representation of $\hbox{SL}_{2}(\bR)\ltimes \bR^{2}$
without $\bR^2$-invariant vectors, its restriction to  
$\hbox{SL}_{2}(\bR)$ has the spectral gap property (see \cite[Corollary~1.4.13]{BHV}).
In particular, it follows that the action of $S_I$ on 
$(\mathcal{H}^{V_i})^\perp$ and on $(\mathcal{H}^{W_j})^\perp$
has the spectral gap property, so that the action of $S_I$
on
$$
\mathcal{H}^0_I:=\big(\oplus_i  (\mathcal{H}_I^{V_i})^\perp\big)\oplus \big(\oplus_j  (\mathcal{H}_I^{W_j})^\perp\big)
$$
has the spectral gap property. Moreover,
$$
(\mathcal{H}^0_I)^\perp =\big(\cap_{i}  \mathcal{H}_I^{V_i}\big)\cap \big(\cap_j \mathcal{H}_I^{W_j}\big)= \mathcal{H}_I^{U_I^{(1)}}\cap \mathcal{H}_I^{U_I^{(2)}}
=\mathcal{H}_I^{U_I}=0.
$$
This completes the proof.
\end{proof}

Using the above results, and following the proof in \cite[Lemma~4.3]{BG2}, we get the following corollary:

\begin{corollary}\label{cor:decor}
	There exist $\ell\ge 1$ and  $\tau>0$ such that 
	for every $\phi,\psi\in C_c^\infty (X_I)$, $u_1,u_2\in\bR$, and $w\in \hbox{\rm Lie}(U\cap S_I)$,
	\begin{align*}
		\int_{X_I} (\phi \circ \exp(u_1 w))\cdot  (\psi \circ \exp(u_2 w))\, d\mu_I = &\left(\int_{X_I}\phi\, d\mu_I\right) \left(\int_{X_I}\psi\, d\mu_I\right)\\ 
		&+ O\left(  \max\big(1,|u_1-u_2|\|w\|\big)^{-\tau} \|\phi\|_{L^2_\ell}\|\psi\|_{L^2_\ell}\right).
	\end{align*}
\end{corollary}

\subsection*{General version of Theorem \ref{th:km}}

Given an admissible $I\subset \{1,\ldots, m+n\}$, we set
\[
A^{+}_I = \big\{ t \in \bR^d_{+} \,  : \,  t_j = 0,  \enskip \textrm{for all $j \notin I$}\big\}.
\]
For $I \subset \{1,\ldots,m+n\}$ and $t\in A^+$, we define 
\[
\lfloor t \rfloor_I := \min_{i \in I} t_i.
\]

The following theorem is a straightforward generalization of Theorem \ref{th:km}.

\begin{theorem}\label{th:km_new}
	Let $I\subset \{1,\ldots, m+n\}$ be admissible subset.
	Then
	there exist $\ell\ge 1$ and $\delta>0$ such that for every $f\in C^\infty (Y)$, 
	$\varphi\in C_c^\infty (X_I)$, and $t\in A_I^+$,
	$$
	\int_Y f\cdot(\varphi\circ a(t))\, d\nu=\left(\int_Y f\, d\nu\right)  	
	\left(\int_{X_I} \varphi\, d\mu_I\right)+O\Big( e^{-\delta \lfloor t\rfloor_I} \|f\|_{C^\ell} \mathcal{N}_\ell(\varphi)\Big),
	$$
where $\mathcal{N}_\ell(\varphi):=\max\big(\|\varphi\|_{C^0},\|\varphi\|_{Lip}, \|\varphi\|_{L^2_\ell} \big).$	
\end{theorem}

\begin{proof}
The proof can be completed as in \cite{KM2} (explicit dependence 
of the functions in this argument was worked out in \cite[Sec.~2]{BG1}).
Without loss of generality, we may assume that $I$ is given as in \eqref{eq:i_0}.
We also note that the space $X_I$ is a bundle
$$
\pi_I:X_I\to S_I/(S_I\cap \Gamma)
$$
with the base
$$
S_I/(S_I\cap \Gamma)\simeq \hbox{SL}_{k_1+k_2}(\bR)/\hbox{SL}_{k_1+k_2}(\bZ)
$$
and compact fibers isomorphic to $U_I/s(U_I\cap \Gamma)s^{-1}$, $s\in S_I$.
We view the base as the space of unimodular lattices in $\bR^{k_1+k_2}$
and define an exhaustion of $X_I$ by compact subsets 
$$
K_\eps:=\{x\in X_I:\, \|v\|\ge \eps \quad\hbox{for all non-zero $v\in \pi_I(x)$} \}.
$$	
Then one has the non-divergence estimate as \cite[Corollary 3.4]{KM2} for the translates $a(t)Ux$ with $t\in A^+_I$ when  $\lfloor t\rfloor_I$
is sufficiently large. It is also clear that for $x\in K_\eps$,
the injectivity radius $\iota(x)$ of $x$ can be estimated as in 
\cite[Proposition 3.5]{KM2}.
Let us also consider a one-parameter sub-semigroup of $A^+_I$
$$
g_t:=\hbox{diag}\big(1,\ldots,1, e^{t/k_1},\ldots,e^{t/k_1}, e^{-t/k_2},\ldots,e^{-t/k_2}, 1,\ldots,1\big).
$$
We observe that 
$$
\hbox{Lie}(G_I)=L\oplus\hbox{Lie}(U_I),
$$ 
where $L$ is the subspace spanned by matrices
\begin{align*}
	&E_{ij},\quad 1\le i\le m+k_2,\; m-k_1+1\le j\le m,\\
	&E_{ij},\quad m+1\le i\le m+k_2,\; m+1\le j\le m+n.
\end{align*}
One easily verifies that the adjoint action $\hbox{Ad}(g_t)X=g_t X g_t^{-1}$
is non-expanding on $L$. Moreover, according to Theorem \ref{th:mixing_I},
there exist $\ell\ge 1$ and $\delta'>0$ such that for every 
$\varphi,\psi\in C_c^\infty (X_I),$ and $t\ge 0$,
$$
\int_{X_I} (\varphi\circ g_t)\cdot \psi\, d\mu_I=\left(\int_{X_I} \varphi\, d\mu_I\right)  	
\left(\int_{X_I} \psi\, d\mu_I\right)+O\Big( e^{-\delta' t}\, \|\varphi\|_{L^2_\ell} \|\psi\|_{L^2_\ell}\Big).
$$
Now we can apply the argument as in \cite[Theorem 2.3]{KM2} to derive quantitative equidistribution of the leaves
$g_tUx$: 
there exists $c,\gamma>0$ such that for every $F\in C_c^\infty (U_I)$
with $\hbox{dist}(\supp(F),e)\le \rho$, $x\in X_I$ with $\iota(x)>2\rho$,
$\varphi\in C_c^\infty (X_I),$ and $t\ge 0$,
\begin{align*}
	\int_{U_I} F(u) \varphi(g_t ux)\, du=&\left(\int_{U_I} F(u)\, du\right)  	
	\left(\int_{X_I} \varphi\, d\mu_I\right)\\
	&\quad\quad\quad+O\Big(
	\rho\|F\|_{L^1}\|\varphi\|_{Lip}+ \rho^{-c} e^{-\gamma t} \|F\|_{C^\ell} \|\varphi\|_{L^2_\ell}\Big).
\end{align*}
With those ingredients the proof can be completed as in
\cite{KM2}.
\end{proof}

\section{Decorrelation estimates} \label{sec:decor}

One of the challenges of our present argument is that 
the decorrelation estimate in Corollary \ref{cor:decor} only holds for one-parameter 
subgroups in directions $w\in\hbox{Lie}(U\cap S_I)$.
Nonetheless, we establish that certain "partial" decorrelation estimates 
also hold for other directions.
A basic prototype of this phenomenon is the following elementary result  
for circle rotations:

\begin{lemma}\label{lem:eq_part}
	For every $\phi,\psi\in L^2(\mathbb{R}/\mathbb{Z})$ and $L\ge 1$,
	\begin{align*}
	\frac{1}{L^2}\int_0^L \int_0^L \int_{\mathbb{R}/\mathbb{Z}} \phi(x+s_1) \overline{\psi(x+s_2)}\,dx \, ds_1ds_2=&
	\left(\int_{\mathbb{R}/\mathbb{Z}} \phi(x)\, dx\right)  \left(\int_{\mathbb{R}/\mathbb{Z}} \overline{\psi(x)}\, dx\right)\\
	 &\quad\quad +O\Big(L^{-2} \|\phi\|_{L^2}\|\psi\|_{L^2}\Big).
	\end{align*}
\end{lemma}

\begin{proof}
	We have the Fourier expansions:
	$$
	\phi(x)= \sum_{n\in\bZ} \widehat \phi(n) e^{2\pi i n x}\quad\hbox{and}\quad \psi(x)= \sum_{n\in\bZ} \widehat \psi(n) e^{2\pi i n x},
	$$
	so that
	$$
	\int_{\bR/\bZ} \phi(x+s_1)\overline{\psi(x+s_2)}\, dx
	=\sum_{n\in \bZ} \widehat \phi(n) \overline{\widehat \psi(n)} e^{2\pi i(s_1-s_2)n}.
	$$
	Since for $n\ne 0$,
	$$
	\frac{1}{L^2} \int_0^L \int_0^L e^{2\pi i(s_1-s_2)n}\, ds_1ds_2\ll n^{-2} L^{-2},
	$$
	it follows that from the Plancherel Formula that
	\begin{align*}
		\frac{1}{L^2}\int_0^L \int_0^L \int_{\mathbb{R}/\mathbb{Z}} \phi(x+s_1) \overline{\psi(x+s_2)}\,dx \, ds_1ds_2=&
\widehat \phi(0) \overline{\widehat \psi(0)}+O(L^{-2})\|\phi\|_{L^2}\|\psi\|_{L^2},		
	\end{align*}
which proves the lemma.
\end{proof}

\subsection{Mean decorrelation for flows - special case}

We will need a version of Lemma \ref{lem:eq_part}
for flows on the space of affine lattices.
Let
\begin{equation}\label{eq:space0}
\cL:=H/\Delta\quad\quad\hbox{with $H:=\hbox{SL}_{d}(\bR)\ltimes\bR^{d}$
and $\Delta:=\hbox{SL}_{d}(\bZ)\ltimes\bZ^{d}$}
\end{equation}
with $d\ge 2$.
We observe that $\cL$ can be identified with the space of affine unimodular lattices in $\mathbb{R}^{d}$:
$$
\cL\simeq \big\{\Lambda+x:\, \hbox{$\Lambda$ -- unimodular lattice in $\mathbb{R}^{d}$},\; x\in \mathbb{R}^d/\Lambda \big\}.
$$
We denote by 
$$
\cL^\star:=\hbox{SL}_{d}(\bR)/\hbox{SL}_{d}(\bZ)
$$
the space of unimodular lattices, and we have the natural projection map
$$
\pi:\cL\to \cL^\star: \Lambda+x\mapsto \Lambda,
$$
whose fibers are tori $\pi^{-1}(\Lambda)\simeq \mathbb{R}^d/\Lambda$.
For $z\in \cL^\star$, we denote $\tau_z$ the probability invariant measure 
supported on the torus $\pi^{-1}(z)$. Then the probability invariant measure
on $\cL$ can be written as 
$$
\mu=\int_{\cL^\star} \tau_z\, d\mu^\star(z),
$$
where $\mu^\star$ denotes the invariant probability measure on $\cL^\star$. For $\phi\in C_c(\cL)$, we define $\phi^\star\in C_c(\cL^\star)$ by
$$
\phi^\star(\Lambda):=\int_{\bR^d/\Lambda} \phi(y+\Lambda)\, d\tau_\Lambda(y+\Lambda).
$$
We think of $\phi^\star$ as a function on $\cL$ which is constant on the fibers. Finally, given a non-zero $w \in \bR^d$, we get a one-parameter flow $\sigma_w : \bR \ra H$ by
\[
\sigma_w(u) = (0,uw) \in H, \quad u \in \bR.
\]
\begin{theorem}\label{th:eq_part_11}
For every $w \in \bR^d \setminus \{0\}$, and for all $\phi,\psi\in C_c^\infty(\cL)$ and $L>0$,
\begin{align*}
\frac{1}{L^2}\int_0^L \int_0^L \int_{\cL}  (\phi\circ \sigma_w(u_1))\cdot & (\bar\psi\circ \sigma_w (u_2))\,d\mu \,  du_1du_2\\
=&
\int_{\cL} (\phi^\star\cdot {\bar\psi^\star})\, d\mu
+O\Big(\max\big(1,L\|w\|\big)^{-2/3} \|\phi\|_{C^\ell} \|\psi\|_{C^\ell}\Big)
\end{align*}
for sufficiently large $\ell$.
\end{theorem}

\begin{proof}
Without loss of generality, we may assume that $\|w\|=1$. For a lattice $\Lambda$ in $\cL^\star$, we write 
$$
\phi^{\Lambda}(x):=\phi(\Lambda+x)\quad\hbox{for $x\in \bR^d/\Lambda$}.
$$
Then 
$$
(\phi\circ \sigma_w(u))^{\Lambda}(x)=\phi^\Lambda(x+uw)\quad\hbox{and}\quad
(\psi\circ \sigma_w(u))^{\Lambda}(x)=\psi^\Lambda(x+uw).
$$
The functions $\phi^{\Lambda}$ and $\psi^{\Lambda}$ have the Fourier expansions:
$$
\phi^{\Lambda}(x)= \sum_{\lambda\in\Lambda^\perp} \widehat \phi^\Lambda(\lambda) e^{2\pi i \left<x,\lambda\right>}\quad\hbox{and}\quad
\psi^{\Lambda}(x)= \sum_{\lambda\in\Lambda^\perp} \widehat \psi^\Lambda(\lambda) e^{2\pi i \left<x,\lambda\right>},
$$
where $\Lambda^\perp$ denotes the dual lattice for $\Lambda$.
Since $\phi$ and $\psi$ are smooth, the series converge uniformly,
and thus
\begin{align*}
\int_{\cL} \phi(\sigma(u_1)z){\bar\psi(\sigma(u_2)z)}\,d\mu(z)
=\int_{\cL^\star} \int_{\bR^d/\Lambda} \phi^\Lambda(x+u_1w){\bar\psi^\Lambda(x+u_2w)} dx \, d\mu^\star(\Lambda),
\end{align*}
and 
$$
\int_{\bR^d/\Lambda} \phi^\Lambda(x+u_1w){\bar\psi^\Lambda(x+u_2w)} dx
=\sum_{\lambda\in\Lambda^\perp} \widehat \phi^\Lambda(\lambda) \overline{\widehat \psi^\Lambda(\lambda)} e^{2\pi i(u_1-u_2)\left<w,\lambda\right>}.
$$
Since 
$$
\frac{1}{L^2} \int_0^L \int_0^L e^{2\pi i(u_1-u_2)\left<w,\lambda\right>}\, du_1du_2
=\omega\big(\pi L \left<w,\lambda\right>\big),	
$$
where $\omega(t)=\frac{\sin^2(t)}{t^2}$ for $t\ne 0$ and $\omega(0)=1$,
we conclude that 
\begin{align*}
&\frac{1}{L^2}\int_0^L \int_0^L  \int_{\bR^d/\Lambda} \phi^\Lambda(x+u_1w){\bar\psi^\Lambda(x+u_2w)} dx \, du_1du_2\\ =&
\sum_{\lambda\in\Lambda^\perp} \widehat \phi^\Lambda(\lambda) \overline{\widehat \psi^\Lambda(\lambda)} 
\omega\big(\pi L \left<w,\lambda\right>\big) \\
=& \phi^\star(\Lambda) \bar\psi^\star(\Lambda)+
\sum_{\lambda\in\Lambda^\perp\backslash \{0\}} \widehat \phi^\Lambda(\lambda) \overline{\widehat \psi^\Lambda(\lambda)} 
\omega\big(\pi L \left<w,\lambda\right>\big).
\end{align*}
Now it remains to estimate the integral of the last sum.
Since for $\phi,\psi\in C_c^\infty(\cL)$, 
\begin{equation}
	\label{eq:fur_bound00}
\big|\widehat \phi^\Lambda(\lambda)\big|\ll_\ell \max\big(1,\|\lambda\|\big)^{-\ell}\,\|\phi\|_{C^\ell}
\quad\hbox{and}\quad
\big|\widehat \psi^\Lambda(\lambda)\big|\ll_\ell \max\big(1,\|\lambda\|\big)^{-\ell}\,\|\psi\|_{C^\ell}
\end{equation}
for every $\ell\ge 1$.
This leads to the estimate
\begin{align*}
\frac{1}{L^2}\int_0^L \int_0^L \int_{\cL} \phi(\sigma(u_1)z){\bar\psi(\sigma(u_2)z)}\,d\mu(z) \, du_1du_2=&
\int_{\cL^\star} (\phi^\star\cdot \bar{\psi^\star})\, d\mu^\star \\
&\quad\quad+O\Big(\Omega(L) \|\phi\|_{C^\ell} \|\psi\|_{C^\ell}\Big),
\end{align*}
where
$$
\Omega(L):=
\int_{\cL^\star} \left(\sum_{\lambda\in\Lambda^\perp\backslash \{0\}} 
\max\big(1,\|\lambda\|\big)^{-2\ell}\omega\big(\pi L \left<w,\lambda\right>\big)\right)\, d\mu^\star(\Lambda).
$$
Since $\phi^\star\cdot \bar{\psi^\star}$ is constant on the fibers,
$$
\int_{\cL^\star} (\phi^\star\cdot \bar{\psi^\star})\, d\mu^\star=
\int_{\cL} (\phi^\star\cdot \bar{\psi^\star})\, d\mu.
$$
By Siegel's Formula \cite{Siegel},
$$
\Omega(L)=\int_{\bR^d} \max\big(1,\|x\|\big)^{-2\ell}\omega\big(\pi L \left<w,x\right>\big)\, dx.
$$ 
Applying an orthogonal change of variables,
we may assume that $w=e_1$. For $C\ge 1$, we obtain
\begin{align*}
\int_{\bR^d} \max\big(1,\|x\|\big)^{-2\ell}\omega\big(\pi L x_1 \big)\, dx
& \le  \int_{|x_1|\le C^{-1}} \max\big(1,\|x\|\big)^{-2\ell}\, dx\\
&\quad\quad+ \frac{C^2}{\pi^2  L^{2}}\int_{|x_1|\ge C^{-1}} \max\big(1,\|x\|\big)^{-2\ell} \, dx\\
&\ll_{\ell} \, C^{-1} + L^{-2}C^2,
\end{align*}
provided that $\ell$ is sufficiently large. Taking $C=L^{2/3}$, we obtain the 
stated estimate.
\end{proof}

\subsection{Mean decorrelation for flows - general version}

Let us now extend the previous result to more general homogeneous spaces.
Let us fix $d\ge 2$ and $d'\ge 1$ and set
\begin{equation}\label{eq:space1}
H:=\left(
\begin{tabular}{ccc}
1 & $\hbox{Mat}_{1,d}(\bR)$ & $\hbox{Mat}_{1,d'}(\bR)$ \\
0 & $\hbox{SL}_{d}(\bR)$ & $\hbox{Mat}_{d,d'}(\bR)$ \\
0 & 0 & $I_{d'}$
\end{tabular}
\right)
\quad\hbox{and}\quad \Delta:=H\cap \hbox{SL}_{d+d'+1}(\bZ).
\end{equation}
We consider the homogeneous space 
\begin{equation}\label{eq:space2}
Z:=H/\Delta.
\end{equation}
We note that the space $\cL$ of affine unimodular lattices, which we introduced above, can be considered as 
a degenerate case of this definition with $d'=0$.\\

As above we have the natural projection map
$$
\pi:Z\to Z^\star,
$$
where $Z^\star$ denotes
the homogeneous space 
of $(d+d')$-matrices obtained
by removing the first row and the first column.
The space $Z^\star$ can be viewed as the spaces of tuples
$$
(\Lambda,v_1,\ldots,v_{d'}),
$$
where $\Lambda$ is a unimodular lattice in $\bR^{d}$, and $v=(v_1,\ldots,v_{d'})\in (\bR^{d}/\Lambda)^{d'}$.
The fibers of this map are tori given by
$$
\pi^{-1}(\Lambda,v)\simeq \bR^{d+d'}/\Lambda_v,
$$
where 
$$
\Lambda_v:=\big\{(\delta, m_1-\delta\cdot v_1,\ldots,m_{d'}-\delta\cdot v_{d'}):\, \delta \in \Lambda^\perp, m\in \bZ^{d'}\big\}.
$$
We note that the dual lattice is given by
$$
\Lambda_v^\perp=\big\{(\delta+m_1v_1+\cdots+m_{d'}v_{d'}, m):\, \delta \in \Lambda, m\in \bZ^{d'}\big\}.
$$
For $z\in Z^\star$, we denote $\tau_z$ the probability invariant measure 
supported on the torus $\pi^{-1}(z)$. Then the probability invariant measure
on $Z$ is 
$$
\mu=\int_Z \tau_z\, d\mu^\star(z),
$$
where $\mu^\star$ denotes the invariant probability measure on $Z^\star$. For $\phi\in C_c(Z)$, we define $\phi^\star\in C_c(Z^\star)$ by
$$
\phi^\star(z):=\int_{Z} \phi\, d\tau_z.
$$
We think of $\phi^\star$ as a function on $Z$ which is constant on the fibers. Finally, given a non-zero $w \in \textrm{Mat}_{1,d}(\bR)$, we get a one-parameter flow 
$\sigma_w : \bR \ra H$ by
\[
\sigma_w(u) = \left(
\begin{tabular}{ccc}
1 & $u w$ & 0 \\
0 & $I_d$ & $0$ \\
0 & 0 & $I_{d'}$
\end{tabular} 
\right),
\quad \textrm{for $u \in \bR$}.
\]

\begin{theorem}\label{th:eq_part0}
For every $w \in \textrm{Mat}_{1,d}(\bR) \setminus \{0\}$ and for all $\phi,\psi\in C_c^\infty(Z)$ and $L>0$,
\begin{align*}
\frac{1}{L^2}\int_0^L \int_0^L \int_{Z}  (\phi\circ \sigma_w(u_1))\cdot & (\bar\psi\circ \sigma_w (u_2))\,d\mu \,  du_1du_2\\
=&
\int_{Z} (\phi^\star\cdot {\bar\psi^\star})\, d\mu
+O\Big(\max\big(1,L\|w\|\big)^{-2/3} \|\phi\|_{C^\ell} \|\psi\|_{C^\ell}\Big)
\end{align*}
for sufficiently large $\ell$.
\end{theorem}

\begin{proof}
For a tuple $(\Lambda,v)$ in $Z^\star$, we write 
$$
\phi^{(\Lambda,v)}(x,y)=\phi|_{\pi^{-1}(\Lambda,v)}(x,y)\quad\hbox{for $(x,y)\in (\bR^{d}\times \bR^{d'})/\Lambda_v$}.
$$
Then 
$$
(\phi\circ \sigma_w(u))^{(\Lambda,v)}(x,y)=\phi^{(\Lambda,v)}(x+uw, y)\quad\hbox{and}\quad
(\psi\circ \sigma_w(u))^{(\Lambda,v)}(x,y)=\psi^{(\Lambda,v)}(x+uw, y).
$$
The function $\phi^{(\Lambda,v)}$ and $\psi^{(\Lambda,v)}$ have the Fourier expansions:
$$
\phi^{(\Lambda,v)}(x,y)= \sum_{\lambda\in\Lambda_v^\perp} \widehat \phi^{(\Lambda,v)}(\lambda) e^{2\pi i \left<(x,y),\lambda\right>}\quad\hbox{and}\quad
\psi^{(\Lambda,v)}(x,y)= \sum_{\lambda\in\Lambda_v^\perp} \widehat \psi^{(\Lambda,v)}(\lambda) e^{2\pi i \left<(x,y),\lambda\right>}.
$$
Since $\phi$ and $\psi$ are smooth, the series converge uniformly.
So we obtain 
\begin{align*}
&\int_{Z} \phi(\sigma_w(u_1)z){\bar\psi(\sigma_w(u_2)z)}\,d\mu(z)\\
=&\int_{Z^\star} \int_{(\bR^{d}\times\bR^{d'})/\Lambda_v} \phi^{(\Lambda,v)}(x+u_1w,y){\bar\psi^{(\Lambda,v)}(x+u_2w,y)} dxdy \, d\mu^\star(\Lambda,v),
\end{align*}
and 
\begin{align*}
\int_{(\bR^{d}\times\bR^{d'})/\Lambda_v} \phi^{(\Lambda,v)}(x+u_1w){\bar\psi^{(\Lambda,v)}(x+u_2w)}\, & dxdy\\
=&\sum_{\lambda\in\Lambda_v^\perp} \widehat \phi^{(\Lambda,v)}(\lambda) \overline{\widehat \psi^{(\Lambda,v)}(\lambda)} e^{2\pi i(u_1-u_2)\left<(w,0),\lambda\right>}.
\end{align*}
Since 
$$
\frac{1}{L^2} \int_0^L \int_0^L e^{2\pi i(u_1-u_2)\left<(w,0),\lambda\right>}\, du_1du_2
=\omega\big(\pi L \left<(w,0),\lambda\right>\big),	
$$
we conclude that 
\begin{align*}
&\frac{1}{L^2}\int_0^L \int_0^L  \int_{(\bR^{d}\times\bR^{d'})/\Lambda_v} \phi^{(\Lambda,v)}(x+u_1w,y){\bar\psi^{(\Lambda,v)}(x+u_2w,y)} dxdy \, du_1du_2\\ =&
\sum_{\lambda\in\Lambda_v^\perp} \widehat \phi^{(\Lambda,v)}(\lambda) \overline{\widehat \psi^{(\Lambda,v)}(\lambda)} 
\omega\big(\pi L \left<(w,0),\lambda\right>\big) \\
=& \phi^\star(\Lambda,v) \bar\psi^\star(\Lambda,v)+
\sum_{\lambda\in\Lambda_v^\perp\backslash \{0\}} \widehat \phi^{(\Lambda,v)}(\lambda) \overline{\widehat \psi^{(\Lambda,v)}(\lambda)} 
\omega\big(\pi L \left<(w,0),\lambda\right>\big).
\end{align*}
Now it remains to estimate the integral of the last sum.
Since the $\phi,\psi\in C_c^\infty(Z)$, 
\begin{equation}
	\label{eq:fur_bound}
\big|\widehat \phi^{(\Lambda,v)}(\lambda)\big|\ll_\ell \max\big(1,\|\lambda\|\big)^{-\ell}\,\|\phi\|_{C^\ell}
\quad\hbox{and}\quad
\big|\widehat \psi^{(\Lambda,v)}(\lambda)\big|\ll_\ell \max\big(1,\|\lambda\|\big)^{-\ell}\,\|\psi\|_{C^\ell}
\end{equation}
for every $\ell\ge 1$, this leads to the estimate
\begin{align*}
\frac{1}{L^2}\int_0^L \int_0^L \int_{Z} \phi(\sigma_w(u_1)z){\bar\psi(\sigma_w(u_2)z)}\,&d\mu(z) \, du_1du_2\\
=&
\int_{Z^\star} (\phi^\star\cdot \bar{\psi^\star})\, d\mu^\star+O\Big(\Omega(L) \|\phi\|_{C^\ell} \|\psi\|_{C^\ell}\Big),
\end{align*}
where
\begin{align*}
\Omega(L):=&
\int_{Z^\star} \left(\sum_{\lambda\in\Lambda_v^\perp\backslash \{0\}} 
\max\big(1,\|\lambda\|\big)^{-2\ell}\omega\big(\pi L \left<(w,0),\lambda\right>\big)\right)\, d\mu^\star(\Lambda,v)\\
=&\int_{Z^\star} \left(\sum_{(\delta,m)\in (\Lambda\times \bZ^{d'})\backslash \{0\}} 
\max\Big(1,\Big\|(\delta+{\sum}_i m_iv_i,m)\Big\|\Big)^{-2\ell}\right. \\
&\quad\quad\quad\quad \quad\quad\quad\quad\quad\quad\quad\quad\quad\quad \times \left. \omega\big(\pi L \big<w,\delta+{\sum}_i m_iv_i\big>\big)\right)\, d\mu^\star(\Lambda,v).
\end{align*}
Since the part of the sum with $m_1=\cdots=m_{d'}=0$ has been handled in the proof of the previous theorem,
it remains to estimate
\begin{align*}
\Omega_1(L)
=&\int_{Z^\star} \left(\sum_{\delta\in \Lambda, m\in \bZ^{d'}\backslash \{0\}} 
\max\Big(1,\Big\|(\delta+{\sum}_i m_iv_i,m)\Big\|\Big)^{-2\ell}\right. \\
&\quad\quad\quad\quad \quad\quad\quad\quad\quad\quad\quad\quad\quad\quad \times \left. \omega\big(\pi L \big<w,\delta+{\sum}_i m_iv_i\big>\big)\right)\, d\mu^\star(\Lambda,v).
\end{align*}
Let us suppose, for example, that $m_1\ne 0$. 
Then we consider the change of variables
$$
(\bR^{d}/\Lambda)^{d'}\to (\bR^{d}/\Lambda)^{d'}: \tilde v_1= \sum_{i=1}^{d'} m_iv_i, \tilde v_2=v_2,\ldots, \tilde v_{d'}=v_{d'}.
$$
We observe that this transformation is 
measure-preserving because it is
a composition
of the maps 
$$
v_1\mapsto m_1 v_1,v_2\mapsto v_2,\ldots, v_{d'}\mapsto v_{d'}\quad\hbox{and}\quad
v_1\mapsto v_1+m_iv_i,v_2\mapsto v_2,\ldots, v_{d'}\mapsto v_{d'}.
$$
Using this the integration reduces to the space $\cL$ of affine lattices in $\bR^d$. Namely,
we obtain 
\begin{align*}
\Omega_1(L)=&
\int_{\cL^\star} \left(\sum_{\delta\in \Lambda, m\in \bZ^{d'}\backslash \{0\}} 
\max\big(1,\|(\delta+\tilde v_1,m)\|\big)^{-2\ell}\omega\big(\pi L \left<w,\delta+\tilde v_1\right>\big)\right)\, d\mu_\cL(\Lambda,\tilde v_1),
\end{align*}
where $\mu_\cL$ denotes the invariant measure on $\cL$.
Applying Siegel's Formula for the space of affine lattices, we obtain 
\begin{align*}
\Omega_1(L)&=\int_{\bR^d} \sum_{m\ne 0} \max\big(1,\|(x,m)\|\big)^{-2\ell}\omega\big(\pi L \left<w,x\right>\big)\, dx\\
&\ll \int_{\bR^d} \sum_{m\ne 0} \big(\|x\|+|m|\big)^{-2\ell}\omega\big(\pi L \left<w,x\right>\big)\, dx\\
&\ll \left(\sum_{m\ne  0} m^{-\ell}\right) \int_{\bR^d} \max\big(1,\|x\|\big)^{-\ell}\omega\big(\pi L \left<w,x\right>\big)\, dx,
\end{align*}
where we used the Cauchy--Schwarz inequality.
Now the estimate can be completed as in the previous proof.
\end{proof}

\medskip

More generally, we derive a version of the previous result for products of spaces as above.
Let $d_1,\ldots,d_k\in \bN$ and $d'_1,\ldots,d'_k\in \bN\cup\{0\}$.
When $d_s\ge 2$ and $d_s'=0$, we define $H_s$, $\Delta_s$, and $Z_s$ as in \eqref{eq:space1}--\eqref{eq:space2}.
When $d_s\ge 2$ and $d_s'=0$, we define $H_s$, $\Delta_s$, and $Z_s$ as in \eqref{eq:space0}.
When $d_s=1$, we set 
$$
H_s:=\bR,\quad \Delta_s:=\bZ,\quad \hbox{and}\quad Z_s:=H_s/\Gamma_s.
$$
We recall that when $d_s\ge 2$ the spaces $Z_s$ are equipped with the invariant probability
measures $\mu_s$, and they fiber over the spaces $Z_s^\star$
which is equipped with the invariant probability measure $\mu_s^\star$.
When $d_s\ge 2$ and $d'_s=0$, $Z_s^\star$ is the space of unimodular  lattices $\Lambda_s$ in $\bR^{d_s}$.
When $d_s\ge 2$ and $d'_s\ge 1$, $Z_s^\star$ is the spaces of tuples
$(\Lambda_s,v_s)$,
where $\Lambda_s$ is a unimodular lattice in $\bR^{d_s}$ and $v_s=(v_{s,1},\ldots,v_{s,d_s'})\in (\bR^{d_s}/\Lambda)^{d_s'}$.
To have consistent notation, we write the elements of $Z_s^\star$ as $(\Lambda_s,v_s)$ in both cases
thinking that $v_s=0$ when $d_s'=0$.
When $d_s=1$, we view $Z_s^\star$
as a one-point space equipped with the Dirac measure $\mu_s^\star$. 

Let 
$$
\hbox{$w_s \in \hbox{Mat}_{1,d_s}(\bR)\subset H_s$, $s=1,\ldots, k$, 
with $\|w_1\|\ge \cdots \ge \|w_k\|$}
$$
and 
$\sigma_s (u)=u w_s$ be
the corresponding one-parameter subgroups.

\begin{theorem}\label{th:eq_part}
	There exists $\delta>0$ such that
	for every $\phi_s,\psi_s \in C_c^\infty(Z_s)$, $s=1,\ldots,k$, and $L>0$,
	\begin{align*}
		\frac{1}{L^2}\int_0^L \int_0^L &\prod_{s=1}^k \int_{Z_s}  (\phi_s\circ \sigma_s(u_1))\cdot (\bar\psi_s\circ \sigma_s (u_2))\,d\mu_s \, du_1du_2\\
		=& \prod_{s=1}^k
		\int_{Z_s} (\phi_s^\star\cdot {\bar\psi_s^\star})\, d\mu_s
		+O\Big(\max\big(1,L\|w_k\|\big)^{-\frac{2}{k(k+2)}} {\prod}_{s=1}^k \|\phi_s\|_{C^\ell} \|\psi_s\|_{C^\ell}\Big)
	\end{align*}
for sufficiently large $\ell$.	
\end{theorem}

\begin{proof}
	As in the previous proof,  we obtain that
	\begin{align*}
	\int_{(\bR^{d_s}\times \bR^{d_s'})/\Lambda_{s,v_s}} \phi^{(\Lambda_s,v_s)}(x_s+u_1w_s,y_s) &{\bar\psi^{(\Lambda_s,v_s)} ( x_s+u_2w_s,y_s)} \, dx_sdy_s\\
	&=\sum_{\lambda_s\in\Lambda_{s,v_s}^\perp} \widehat \phi^{(\Lambda_s,v_s)}(\lambda_s) \overline{\widehat \psi^{(\Lambda_s,v_s)}(\lambda_s)} e^{2\pi i(u_1-u_2)\left<(w_s,0),\lambda_s\right>}.
	\end{align*}
	Then 
	\begin{align*}
	\frac{1}{L^2}\int_0^L \int_0^L \prod_{s=1}^k \int_{(\bR^{d_s}\times \bR^{d_s'})/\Lambda_{s,v_s}} \phi^{(\Lambda_s,v_s)}(x_s+u_1w_s,y_s){\bar\psi^{(\Lambda_s,v_s)}(x_s+u_2w_s,y_s)} dx_sdy_s\,du_1 du_2\\
	=\sum_{\lambda_1,\ldots, \lambda_s} \prod_{s=1}^k \widehat \phi^{(\Lambda_s,v_s)}(\lambda_s) \overline{\widehat \psi^{(\Lambda_s,v_s)}(\lambda_s)}\,
    \omega\Big(\pi L {\sum}_{j=1}^k \left<(w_j,0),\lambda_j\right> \Big),
	\end{align*}
    where the sum is taken over $\lambda_1\in \Lambda_{1,v_1}^\perp$, \ldots,  $\lambda_k\in \Lambda_{k,v_k}^\perp$.
	Since the function $\omega$ is uniformly bounded,
	in view of estimates \eqref{eq:fur_bound}, this sum converges uniformly. We note that the term with $\lambda_1=\cdots=\lambda_k=0$ gives
	$$
	\prod_{s=1}^k \int_{Z_s^\star} \widehat \phi^{\Lambda_s}(0) \overline{\widehat \psi^{\Lambda_s}(0)}\, d\mu_s^{\star}(\Lambda_s,v_s)
	=\prod_{s=1}^k
	\int_{Z^\star_s} (\phi_s^\star\cdot {\bar\psi_s^\star})\, d\mu_s^\star=\prod_{s=1}^k
	\int_{Z_s} (\phi_s^\star\cdot {\bar\psi_s^\star})\, d\mu_s,
	$$
	which is the main term in the theorem. Hence, it remains to estimate
	the integrals of remaining summands over 
	$(\lambda_1,\ldots,\lambda_k)\ne (0,\ldots,0)$.
	In fact, we only discuss the sum over $\lambda_1,\ldots,\lambda_k\ne 0$
	because the remaining sums reduces to similar estimates 
	with fewer number of factors. Now we need to estimate the expression
	\begin{align*}
	\int_{Z_1^\star\times\cdots\times  Z_k^\star}{\sum_{\lambda_1,\ldots,\lambda_k}}^{\!\!\!\!\!*}\, 
	\prod_{s=1}^k \max & \big(1,\|\lambda_s\|\big)^{-2\ell}\\
    &\times  
	\omega\Big(\pi L {\sum}_{j=1}^k \left<(w_j,0),\lambda_j\right> \Big)\, d\mu_1^\star(\Lambda_1,v_1)\cdots d\mu_k^\star(\Lambda_k,v_k),
	\end{align*}
	where the sum is taken $\lambda_1\in \Lambda_{1,v_1}^\perp\backslash \{0\},\ldots, \lambda_k\in \Lambda_{k,v_k}^\perp\backslash \{0\}$.
    More explicitly, when $d_s'\ge 1$,  the parameters $\lambda_s$ run over
    $$
    \lambda_s=(\delta_s+m_{s,1}v_{s,1}+\cdots+m_{s,d'}v_{s,d'}, m_s)\quad \hbox{with $\delta_s \in \Lambda_s$ and $m_s\in \bZ^{d_s'}$.}
    $$
    One can handle the sum over $\lambda_s$ as in the previous proof using a change of variables
    in $(v_{s,1},\ldots, v_{s,d_s'})$ and the Siegel formula for the affine lattices.
    When $d_s\ge 2$ and $d_s'=0$, the parameter $\lambda_s$ runs over $\Lambda_s$,
    and one can apply the Siegel formula for the space of lattices.
    Ultimately, this gives the integral
	\begin{equation}\label{eq:ii}
	\int_{\bR^{d_1}\times \cdots\times  \bR^{d_k}}
	\prod_{s=1}^k \max\big(1,\|x_s\|\big)^{-\ell}\,
	\omega\Big(\pi L {\sum}_{j=1}^k \left<w_j,x_j\right> \Big)\, d\rho_1(x_1)\cdots d\rho_k(x_k),
	\end{equation}
	where $\rho_s$ denotes the Lebesgue measure on $\bR^{d_s}$ when $d_s\ge 2$
	and $\rho_s=\sum_{z\in\bZ\backslash \{0\}}\delta_z$ when $d_s=1$.
	We write $x_j=(x_j^{(1)},\ldots,x_j^{(d_j)})$.
	Applying orthogonal changes of variables in \eqref{eq:ii}, 
	we may assume that 
	$$
	\left<w_j,x_j\right>=\|w_j\|\cdot |x_j^{(1)}| \quad \hbox{for $j=1,\ldots, k$.}
	$$ 
	
	When $d_s=1$ for some $s$, we observe that for $\max_j |x_j|\ge 1$,
	$$
	\omega\Big(\pi L {\sum}_{j=1}^k \|w_j\|\cdot |x_j| \Big)\ll (L\|w_k\|)^{-2}.
	$$
	This also implies the bound $O_\ell \big((L\|w_k\|)^{-2}\big)$ for the integral \eqref{eq:ii}.
	
	Now let us assume that $d_s\ge 2$ for all $s=1,\ldots k$. 
	Let $C\ge 1$.
	When we integrate over the domain defined by $|x_j^{(1)}|\le C^{-1}$ for all $j=1,\ldots,k$, we obtain the bound:
	\begin{align*}
	\ll & \prod_{s=1}^k
	\int_{|x_s^{(1)}|\le C^{-1}}
	 \max\big(1,\|x_s\|\big)^{-\ell}\, d\rho_s(x_s)\ll_\ell C^{-k}
	\end{align*}
	when $\ell$ is sufficiently large.
	Also to estimate the integrals over the domains definded by 
	$|x_j^{(1)}|\ge C^{-1}$ for some $j$, we use the bound:
	$$
	\omega\Big(\pi L {\sum}_{j=1}^k \|w_j\|\cdot |x_j| \Big)\le
	\frac{C^2}{\pi^2  L^2\|w_j\|^2},
	$$
	which leads to the bound
	\begin{align*}
		\int_{\bR^{d_s}}
		\max\big(1,\|x_s\|\big)^{-\ell}\omega\Big(\pi L {\sum}_{j=1}^k \|w_j\|\cdot |x_j| \Big)\, d\rho_s(x_s)
		\ll_\ell C^2(L\|w_k\|)^{-2}
	\end{align*}
	provided that $\ell$ is sufficiently large. 
	Ultimately, we conclude that the integral \eqref{eq:ii} is bounded up to a constant by 
	$$
	C^{-k}+ C^2(L\|w_k\|)^{-2}.
	$$
	with a parameter $C\ge 1$. Optimizing in $C$, we get the required bound.
\end{proof}

A simple modification of the previous argument also gives:

\begin{corollary}\label{cor:eq_part}
	There exists $\delta>0$ such that
	for every $\xi\in \bR$, $\phi_s,\psi_s \in C_c^\infty(Z_s)$, $s=1,\ldots,k$, and $L>0$,
	\begin{align*}
		\frac{1}{L^2}\int_0^L \int_0^L e^{i(u_1-u_2)\xi} &\Big(\prod_{s=1}^k \mu_s\big((\phi_s\circ \sigma_s(u_1))\cdot (\bar\psi_s\circ \sigma_s (u_2))\big)
  -
  \prod_{s=1}^k
		\mu_s (\phi_s^\star\cdot {\bar\psi_s^\star}) \Big) 
   \, du_1du_2\\
		=& O\Big(\max\big(1,L\|w_k\|\big)^{-\frac{2}{k(k+2)}} \,{\prod}_{s=1}^k \|\phi_s\|_{C^\ell} \|\psi_s\|_{C^\ell}\Big)
	\end{align*}
uniformly on $\xi$, for sufficiently large $\ell$.
\end{corollary}

\begin{proof}
It turns out that the previous proof requires only minor modifications.
As before, we obtain
\begin{align*}
	\frac{1}{L^2}\int_0^L \int_0^L e^{i(u_1-u_2)\xi} \Big(\prod_{s=1}^k & \int_{(\bR^{d_s}\times \bR^{d_s'})/\Lambda_{s,v_s}^\perp}  \phi^{(\Lambda_s,v_s)}(x_s+u_1w_s,y_s)\\
 &\times {\bar\psi^{(\Lambda_s,v_s)}(x_s+u_2w_s,y_s)}\, dx_sdy_s
 	 -\prod_{s=1}^k  \widehat \phi^{\Lambda_s}(0) \overline{\widehat \psi^{\Lambda_s}(0)}\Big)
du_1 du_2\\
	=&{\sum_{\lambda_1,\ldots, \lambda_s}}^{\!\!\!\!\!\prime} \,\prod_{s=1}^k \widehat \phi^{\Lambda_s}(\lambda_s) \overline{\widehat \psi^{\Lambda_s}(\lambda_s)}\, \omega\Big(\pi L {\sum}_{j=1}^k \left<(w_j,0),\lambda_j\right>+\xi \Big),
\end{align*}
where the sum is taken over $(\lambda_1,\ldots,\lambda_s)\ne (0,\ldots,0)$.
Using the same argument as above, we obtain the integrals
$$
	\int_{\bR^{d_1}\times \cdots\times  \bR^{d_{k_1}}}
	\prod_{s=1}^{k_1} \max\big(1,\|x_s\|\big)^{-\ell}\,
	\omega\Big(\pi L {\sum}_{j=1}^{k_1} \left<w_j,x_j\right> +\xi\Big)\, d\rho_1(x_1)\cdots d\rho_{k_1}(x_{k_1}),
$$
with $k_1\le k$.
Applying orthogonal changes of variables, we may assume that 
$$
\left<w_j,x_j\right>=\hbox{sign}(\xi)\cdot \|w_j\|\cdot |x_j^{(1)}| \quad \hbox{for $j=1,\ldots, k_1$,}
$$ 
so that we obtain the integral
$$
\int_{\bR^{d_1}\times \cdots\times  \bR^{d_{k_1}}}
\prod_{s=1}^{k_1} \max\big(1,\|x_s\|\big)^{-\ell}\,
\omega\Big(\pi L {\sum}_{j=1}^{k_1} \|w_j\|\cdot |x_j| +|\xi|\Big)\, d\rho_1(x_1)\cdots d\rho_{k_1}(x_{k_1}).
$$
Then the same argument works.
\end{proof}

\begin{remark} \label{rem:general}
Let us consider a slight generalization of the above results.
For $d\ge 2$ and $d_0, d'\ge 1$, we set
\begin{equation}\label{eq:space1_1}
H:=\left(
\begin{tabular}{ccc}
$I_{d_0}$ & $\hbox{Mat}_{d_0,d}(\bR)$ & $\hbox{Mat}_{d_0,d'}(\bR)$ \\
0 & $\hbox{SL}_{d}(\bR)$ & $\hbox{Mat}_{d,d'}(\bR)$ \\
0 & 0 & $I_{d'}$
\end{tabular}
\right)
\quad\hbox{and}\quad \Delta:=H\cap \hbox{SL}_{d+d'+d_0}(\bZ)
\end{equation}
and consider the homogeneous space 
\begin{equation}\label{eq:space2_1}
Z:=H/\Delta.
\end{equation}
Let $Z^\star$ denote the homogeneous space 
obtained from $Z$ by removing the first $d_0$ rows and the first $d_0$ columns.
In fact, the space $Z^\star$ is exactly as in Theorem \ref{th:eq_part0}. Namely, it can be parameterized
by tuples $(\Lambda,v)$ where $\Lambda$ is a unimodular lattice in $\bR^{d}$, and $v\in (\bR^{d}/\Lambda)^{d'}$.
We have the projection map
$$
\pi:Z\to Z^\star
$$
with the fibers
\begin{equation}\label{eq:fibb}
\pi^{-1}(\Lambda,v)\simeq (\bR^{d+d'}/\Lambda_v)^{d_0}.
\end{equation}
Let us take a non-zero element $w\in \hbox{Mat}_{d_0,d}(\bR)\subset H$ such that 
all non-zero components of $w$ are contained in the $i_0$-th row.
We denote by $U_w$ the subgroup of $H$ whose Lie algebra is generated by $E_{i_0,j}$, $j=d_0+1,\ldots,d+d'+d_0$.
We note that $U_w$ is normal in $H$.
The homogeneous subspace $Z_w:=U_w\Delta$ of $Z$ is a $(d+d')$-dimensional torus and 
we denote by $\tau_w$ the probability invariant measure on $Z_w$.
For $\phi\in C_c(Z)$, we define $\phi^\star\in C_c(Z)$ by
$$
\phi^\star(h\Delta):=\int_{Z_w} \phi(hz)\, d\tau_w(z), \quad h\in H.
$$
In this set-up, we can apply the above arguments, in particular,  decompose 
the restriction of the functions to the fibers \eqref{eq:fibb} into Fourier series on $(\bR^{d+d'}/\Lambda_v)^{d_0}$.
Because of our choice of $w$, only the $i_0$-component is affected by the flow,
so that the previous computations require only minor modifications,
and Theorems \ref{th:eq_part0}--\ref{th:eq_part} and Corollary \ref{cor:eq_part} can be proved in this setting as above.

We also note that one can take a non-zero element $w\in \hbox{Mat}_{d,d'}(\bR)\subset H$ such that 
all non-zero components of $w$ are contained in the $j_0$-th row. Clearly, our argument applies in this case as well.
\end{remark}


\section{A proposition about weights} \label{sec:weight}

Let $I\subset \{1,\ldots,m+n\}$ be an admissible subset, i.e. the sets
$$
I^\prime:=I\cap \{1,\ldots,m\}\quad \hbox{and}\quad I^{\prime\prime}:=I\cap \{m+1,\ldots,m+n\}
$$
are both not-empty. Note that $I=I^\prime\sqcup I^{\prime\prime}$.
We define the subsemigroup of $A^+$ as
$$
A^+_I:=\{t\in A^+:\, t_i=0\hbox{ for $i\notin I$}\}.
$$
We consider the adjoint action of $A^+$ on $\hbox{Lie}(U)$.
Then for $t\in A^+$ and the elementary matrices $E_{ij}$ generating $\hbox{Lie}(U)$,
$$
[t,E_{ij}]=\alpha_{ij}(t)E_{ij},
$$
where the weights $\alpha_{ij}$ are given by  $\alpha_{ij}(t)=t_i+t_j$.\\

We say that the weight $\alpha_{ij}$ (or the corresponding weight space $\bR E_{ij}$, or a non-zero vector there)
is \emph{mixing for $I$} if $i\in  I^\prime$ and $j\in  I^{\prime\prime}$.
We note that this condition is equivalent to $\bR E_{ij}\subset \hbox{Lie}(S_I)$, and thus Theorem \ref{th:mixing_I} should explain the terminology.

\begin{proposition}\label{p:weights}
	Let $I_1\ne I_2$ be admissible subsets. Then there exists $\theta>0$ such that 
	for every $s\in A^+_{I_1}$ and $t\in A^+_{I_2}$, either
	there exists a weight $\alpha$ mixing for $I_1$ such that 
	\begin{equation}\label{eq:case1}
		\alpha(s-t)\ge \theta\, \min\big(\lfloor s\rfloor_{I_1}, \lfloor t\rfloor_{I_2}\big),
	\end{equation}
	or 
	there exists a weight $\alpha$ mixing for $I_2$ such that 
	\begin{equation}\label{eq:case2}
		\alpha(t-s)\ge \theta\, \min\big(\lfloor s\rfloor_{I_1}, \lfloor t\rfloor_{I_2}\big).
	\end{equation}
\end{proposition}

\begin{proof}
	For simplicity we write 
    $$
    \Pi(s,t):=\min\big(\lfloor s\rfloor_{I_1}, \lfloor t\rfloor_{I_2}\big)
    $$ 
    and use the following notations:
	\begin{align*}
		D_1 \prec_\eps D_2 \;&\Longleftrightarrow \; D_1\le D_2+ \eps\, \Pi(s,t),\\
		D_1 \succ_\eps D_2 \;&\Longleftrightarrow \; D_1\ge D_2- \eps\, \Pi(s,t),\\
   		D_1 \approx_\eps D_2\;& \Longleftrightarrow\; |D_1-D_2|\le \eps\, \Pi(s,t).
	\end{align*}
	The proof is subdivided into several cases:
	
	\vspace{0.2cm}
	
	\noindent {\sc Case 1:} {\it $I_1^\prime \backslash I_2^\prime\ne \emptyset$ and $I_1^{\prime\prime} \backslash I_2^{\prime\prime}\ne \emptyset$}.\\

	We take $i\in I_1^\prime \backslash I_2^\prime$ and $j\in I_1^{\prime\prime} \backslash I_2^{\prime\prime}$. Then
	$$
	\alpha_{ij}(s-t)=s_i+s_j,
	$$
	which immediately implies the proposition.
	
	\vspace{0.2cm}

	\noindent {\sc Case 2:} {\it $I_2^\prime \backslash I_1^\prime\ne \emptyset$ and $I_2^{\prime\prime} \backslash I_1^{\prime\prime}\ne \emptyset$}. \\

This can be handled as in {\sc Case 1.} \\

\noindent The complement of {\sc Cases 1--2} can be subdivided into the following sub-cases:
 	\begin{itemize}
		\item {\sc Case 3:} $I_1^\prime \backslash I_2^\prime= \emptyset$ and $I_2^{\prime\prime} \backslash I_1^{\prime\prime}= \emptyset$ $\Longleftrightarrow$ $I_1^\prime \subseteq I_2^\prime$ and $I_2^{\prime\prime} \subseteq I_1^{\prime\prime}$,
		\item {\sc Case 4:} $I_1^\prime \backslash I_2^\prime= \emptyset$ and $I_2^\prime \backslash I_1^\prime= \emptyset$ $\Longleftrightarrow$  $I_1^\prime =I_2^\prime$, 
		\item {\sc Case 5:} $I_2^\prime \backslash I_1^\prime= \emptyset$  and $I_1^{\prime\prime} \backslash I_2^{\prime\prime}= \emptyset$  $\Longleftrightarrow$ $I_2^\prime \subseteq I_1^\prime$ and $I_1^{\prime\prime} \subseteq I_2^{\prime\prime}$,
		\item {\sc Case 6:} $I_1^{\prime\prime} \backslash I_2^{\prime\prime}= \emptyset$ and $I_2^{\prime\prime} \backslash I_1^{\prime\prime}= \emptyset$ $\Longleftrightarrow$ $I_1^{\prime\prime} = I_2^{\prime\prime}$. 
	\end{itemize}
	We only discuss {\sc Cases 3--4} since the remaining cases can be treated similarly.\\
	
	Let us fix $\eps >0$, to be specified later, chosen independently of $s$ and $t$, but will depend on $m$ and $n$.

	\vspace{0.2cm}
	
	\noindent {\sc Case 3:} {\it $I_1^\prime \subseteq I_2^\prime$ and $I_2^{\prime\prime} \subseteq I_1^{\prime\prime}$}.\\

 We may additionally assume that $I_1^\prime \subsetneq I_2^\prime$ and $I_2^{\prime\prime} \subsetneq I_1^{\prime\prime}$ because
 otherwise we land either in {\sc Case 4} or {\sc Case 6}. \\
	
	
	Take $i\in I_1^\prime$ and $k\in I_1^{\prime\prime}\backslash I_2^{\prime\prime}$.
	If 
	$(s_i+s_k)-t_i\ge \eps\, \Pi(s,t)$, we can take the weight $\alpha_{ik}$ to obtain \eqref{eq:case1}, so that we may assume that 
	\begin{equation}\label{eq:12}
		s_i+s_k\prec_\eps t_i\quad \hbox{for all $i\in I_1^\prime$ and $k\in I_1^{\prime\prime}\backslash I_2^{\prime\prime}$}.
	\end{equation}
	
	Take $l\in I_2^\prime\backslash I_1^\prime$ and $j\in I_2^{\prime\prime}$.
	If 
	$(t_l+t_j)-s_j\ge \eps\, \Pi(s,t)$, we can take the weight $\alpha_{lj}$ to obtain \eqref{eq:case2}, so that we may assume that 
	\begin{equation}\label{eq:13}
		t_l+t_j\prec_\eps s_j\quad \hbox{for all $l\in I_2^\prime\backslash I_1^\prime$ and $j\in I_2^{\prime\prime}$}.
	\end{equation}
	
	When $t_j\succ_\eps s_j$ for some $j\in I_2^{\prime\prime}$, then 
	it follows from \eqref{eq:12} that for $i\in I_1^\prime$ and $k\in I_1^{\prime\prime}\backslash I_2^{\prime\prime}$, 
	$$
	(t_i+t_j)-(s_i+s_j)\succ_\eps s_k+(t_j-s_j) \succ_\eps s_k\ge \Pi(s,t).
	$$
	Taking $\eps<1/3$, we obtain \eqref{eq:case2} with the weight $\alpha_{ij}$.
	Now we assume that 
	\begin{equation}\label{eq:14}
		t_j\le s_j-\eps\, \Pi(s,t)\quad\hbox{for all $j\in I_2^{\prime\prime}$.}
	\end{equation}
	
	When $s_i\succ_\eps t_i$ for some $i\in I_1^{\prime}$, then 
	it follows from \eqref{eq:13} that for $j\in I_2^{\prime\prime}$ and $l\in I_2^{\prime}\backslash I_1^{\prime}$, 
	$$
	(s_i+s_j)-(t_i+t_j)\succ_\eps t_l+(s_i-t_i) \succ_\eps t_l\ge \Pi(s,t).
	$$
	Taking $\eps<1/3$, we obtain \eqref{eq:case1}.
	Now we assume that 
	\begin{equation}\label{eq:15}
		s_i\le t_i-\eps\, \Pi(s,t)\quad\hbox{for all $i\in I_1^{\prime}$.}
	\end{equation}
	
	\noindent From \eqref{eq:14}, we obtain 
	\begin{equation}\label{eq:16}
		{\sum}_{j\in I_2^{\prime\prime}} t_j \le  {\sum}_{j\in I_2^{\prime\prime}} s_j-|I_1^{\prime\prime}|\eps\, \Pi(s,t)
	\end{equation}
	From \eqref{eq:15}, we obtain 
	\begin{equation}\label{eq:17}
		{\sum}_{i\in I_1^{\prime}} s_i \le  {\sum}_{i\in I_1^{\prime}} t_i-|I_1^{\prime}|\eps\, \Pi(s,t).
	\end{equation}
	Finally, we use that
	$$
	{\sum}_{i\in I_1^{\prime}} s_i={\sum}_{j\in I_1^{\prime\prime}} s_j
	\quad\hbox{and}\quad
	{\sum}_{i\in I_2^{\prime}} t_i={\sum}_{j\in I_2^{\prime\prime}} t_j.
	$$
	Then it follows from \eqref{eq:16} and \eqref{eq:17} that 
	$$
	{\sum}_{i\in I_2^{\prime}} t_i \le {\sum}_{i\in I_1^{\prime}} t_i -(|I_1^{\prime\prime}|+|I_1^{\prime}|)\eps\, \Pi(s,t).
	$$
    Since $I_1^\prime \subsetneq I_2^\prime$, this is only possible when $t_i=0$ for $i\in I_2^\prime\backslash I_1^\prime$ and $\Pi(s,t)=0$.
    In this case, the proposition holds trivially. \\

	\noindent {\sc Case 4:} {\it $I_1^\prime =I_2^\prime$}. \\

	Take $i\in I_1^\prime$ and $j\in I_1^{\prime\prime}\cap I_2^{\prime\prime}$. If 
	$(s_i+s_j)-(t_i+t_j)\ge \eps\, \Pi(s,t)$, we can take the weight $\alpha_{ij}$ to obtain \eqref{eq:case1}. Also if 
	$(t_i+t_j)-(s_i+s_j)\ge \eps\, \Pi(s,t)$, we can take the weight $\alpha_{ij}$ to obtain \eqref{eq:case2}. Hence, it remains to consider the case when
	\begin{equation}\label{eq:11b}
		s_i+s_j \approx_\eps t_i+t_j \quad \hbox{for all $i\in I_1^\prime$ and $j\in I_1^{\prime\prime}\cap I_2^{\prime\prime}$}.
	\end{equation}
	
	Take $i\in I_1^\prime$ and $k\in I_1^{\prime\prime}\backslash I_2^{\prime\prime}$.
	If 
	$(s_i+s_k)-t_i\ge \eps\, \Pi(s,t)$, we can take the weight $\alpha_{ik}$ to obtain \eqref{eq:case1}, so that we may assume that 
	\begin{equation}\label{eq:12b}
		s_i+s_k\prec_\eps t_i\quad \hbox{for all $i\in I_1^\prime$ and $k\in I_1^{\prime\prime}\backslash I_2^{\prime\prime}$}.
	\end{equation}
	
	Take $i\in I_1^\prime$ and $l\in I_2^{\prime\prime}\backslash I_1^{\prime\prime}$.
	If 
	$(t_i+t_l)-s_i\ge \eps\, \Pi(s,t)$, we can take the weight $\alpha_{il}$ to obtain \eqref{eq:case2}, so that we may assume that 
	\begin{equation}\label{eq:13b}
		t_i+t_l\prec_\eps s_i\quad \hbox{for all $i\in I_1^\prime$ and $l\in I_2^{\prime\prime}\backslash I_1^{\prime\prime}$}.
	\end{equation}
	
 Let us suppose that $I_1^{\prime\prime}\backslash I_2^{\prime\prime}\ne\emptyset$ and 
	$I_2^{\prime\prime}\backslash I_1^{\prime\prime}\ne\emptyset$.
We write
$$
\Sigma_1:={\sum}_{i\in I_1'} s_i={\sum}_{j\in I_1''} s_j\quad\hbox{and}\quad \Sigma_2:={\sum}_{i\in I_2'} t_i={\sum}_{j\in I_2''} t_j.
$$
Recall that $I_1'=I_2'$. 
Let $k\in I_1^{\prime\prime}\backslash I_2^{\prime\prime}$ and 
	$l\in I_2^{\prime\prime}\backslash I_1^{\prime\prime}$.
Summing \eqref{eq:12b} and \eqref{eq:13b}, we obtain:
$$
\Sigma_1+|I_1'|\,s_k \prec_{n\eps} \Sigma_2\quad\hbox{and}\quad \Sigma_2+ |I_1'|\,t_l\prec_{n\eps} \Sigma_1.
$$
Hence, it follows that 
$$
|I_1'|\,s_k + |I_1'|\,t_l\prec_{2n\eps} 0.
$$
Taking $\eps <1/(2n)$, we conclude that that $s_k=t_l=0$.
Then $\Pi(s,t)=0$ and the proposition follows trivially. \\

Now we can assume that 
either $I_1^{\prime\prime}\subset I_2^{\prime\prime}$ and $I_2^{\prime\prime}\subset I_1^{\prime\prime}$.
We only discuss the case when $I_1^{\prime\prime}\subset I_2^{\prime\prime}$ because 
the other case can be treated similarly.
Since $I_1\ne I_2$ and $I'_1=I'_2$, we also have $I_1^{\prime\prime}\subsetneq I_2^{\prime\prime}$.
It follows from \eqref{eq:11b} that
$$
|I_1^{\prime\prime}|\sum_{i\in I_1^\prime} s_i+|I_1^\prime|\sum_{j\in I_1^{\prime\prime}} s_j \approx_{mn\eps} |I_1^{\prime\prime}|\sum_{i\in I_1^\prime}t_i+|I_1^\prime|\sum_{j\in I_1^{\prime\prime}}t_j .
$$
Therefore,
\begin{equation}\label{eq:ll1}
\Sigma_1 \approx_{mn\eps} \frac{|I_1^{\prime\prime}|}{|I_1|}\sum_{i\in I_1^\prime}t_i+\frac{|I_1^{\prime}|}{|I_1|}\sum_{j\in I_1^{\prime\prime}}t_j=\Sigma_2-\frac{|I_1^{\prime}|}{|I_1|}\sum_{j\in I_2^{\prime\prime}\backslash I_1^{\prime\prime}}t_j.
\end{equation}
Also from \eqref{eq:13b}, we derive that
\begin{equation}\label{eq:ll2}
\sum_{i\in I_1^\prime} t_i+|I_1^{\prime}| t_l\prec_{n\eps} \sum_{i\in I_1^\prime}s_i=\Sigma_1 
\end{equation}
for $l\in I_2^{\prime\prime}\backslash I_1^{\prime\prime}$.
Combining \eqref{eq:ll1} and \eqref{eq:ll2}, we conclude that
$$
\Sigma_2+|I_1^{\prime}| t_l\prec_{(m+1)n\eps} 
\Sigma_2-\frac{|I_1^{\prime}|}{|I_1|}\sum_{j\in I_2^{\prime\prime}\backslash I_1^{\prime\prime}}t_j,
$$
so that 
$$
t_\ell
\prec_{(m+1)n\eps} 0.
$$
When we take $\eps< 1/((m+1)n)$, this implies that
$t_\ell=0$ for $\ell\in I_2^{\prime\prime}\backslash I_1^{\prime\prime}$.
Then $\Pi(s,t)=0$ and the proposition follows trivially.
This completes analysis of {\sc Case 4}. Since the other cases
can be treated similarly, this concludes the proof
of the proposition.
\end{proof}

\section{Multiple equidistribution} \label{sec:multi}

In this section, we prove the following multiple equidistribution result, which generalizes Theorem \ref{th:bg}.

\begin{theorem}
	\label{th:multi_gen}
	Let $I_1,\ldots,I_r\subset \{1,\ldots,m+n\}$
 be admissible subsets.
 There exist $\ell\ge 1$ and $\delta_r>0$ such that 
	for every $f\in C_c^\infty(Y)$, $\varphi_1,\ldots,\varphi_r\in C^\infty_c(X)$, and $t_1 \in A^{+}_{I_1}$, $\ldots$, $t_r \in A^{+}_{I_r}$,  
	\begin{align*}
	\int_Y f \left(\prod_{s=1}^r \varphi_s \circ a(t_s)\right)\, d\nu=&
	\left( \int_Y f\, d\nu \right)
	 \prod_{s=1}^r \int_{X_{I_s}} \varphi_s \,  d\mu_{I_s} \\
	&\quad+O\left(e^{-\delta_r \min\big(\lfloor t_1 \rfloor_{I_1},\ldots,  \lfloor t_r \rfloor_{I_r},\Delta(t_1,\ldots,t_r)\big)} \,  \|f\|_W{\prod}_{s=1}^r \|\varphi_s\|_{C^\ell} \right).
	\end{align*}
\end{theorem}

To simplify notation, we write
$$
D(t_1,\ldots,t_r):=\min\big(\lfloor t_1 \rfloor_{I_1},\ldots,  \lfloor t_r \rfloor_{I_r},\Delta(t_1,\ldots,t_r)\big).
$$

When $I_1=\cdots =I_r$, Theorem \ref{th:multi_gen} can be proved using the method of \cite{BG2} with minor modifications.
Indeed, let us consider the span $W_I$ of mixing directions $E_{ij}$, namely,
with indices satisfying $i\in I\cap \{1,\ldots,m\}$ and $j\in I\cap \{m+1,\ldots,m+n\}$.
Since the adjoint action of $A_I$ on $W_I$ is proper,
we can run exactly the same argument as in  \cite{BG2}
by considering only one-parameter subgroups $\{\exp(uw)\}_{u\in\mathbb{R}}$ with $w\in W_I$.
Since those one-parameter subgroups satisfy the required mixing estimate
by Theorem \ref{th:mixing_I}, the original argument is applicable.
The base case with $r=1$, which is needed to start this argument, was handled in Theorem \ref{th:km_new}.

\medskip

Now we suppose that some of the sets $I_{s}$ are distinct.
Then by Proposition \ref{p:weights}, there exist
a constant $\theta>0$,  $s_1\ne s_2\in \{1,\ldots,r\}$, and a direction $w=E_{i_0j_0}\in\hbox{Lie}(U)$ with the weight $\alpha=\alpha_{i_0j_0}$ such that 
\begin{equation}
	\label{eq:gap}
\alpha(t_{s_2}-t_{s_1})\ge \theta \, D(t_1,\ldots.t_r)
\end{equation}
and $w$ is mixing for $I_{s_2}$. We change the indexation so that 
\begin{equation}
	\label{eq:aaa}
\alpha(t_1)\ge \alpha(t_2)\ge \cdots \ge \alpha(t_r).
\end{equation}
Then $s_1>s_2$. It follows from \eqref{eq:gap} that
there exists $k=s_2,\ldots, s_1-1$ such that 
\begin{equation}
	\label{eq:gap2}
\alpha(t_{k}-t_{k+1})\ge \theta \, D(t_1,\ldots.t_r)/r.
\end{equation}
From now on, we fix the direction $w=E_{i_0j_0}$ and the index $k$.
It is important for our argument later that for some $s=1,\ldots,k$,
the direction $w$ is mixing for $I_s$. 

\medskip

For smooth functions $f$ on $Y$, we have the uniformly convergent expansion as in \eqref{eq:wiener}:
\[
f = \sum_{\xi \in \Xi_Y} \widehat f(\xi) \xi.
\]
Then
\begin{align}
	\int_Y f\left(\prod_{s=1}^r \varphi_s \circ a(t_s)\right)\, d\nu
	=&\left(\int_Y f\, d\nu\right) \int_Y\left(\prod_{s=1}^r \varphi_s \circ a(t_s)\right)\, d\nu \nonumber\\
	&\quad\quad\quad+\sum_{\xi\ne 1} \widehat f(\xi) \int_Y\xi \left(\prod_{s=1}^r \varphi_s \circ a(t_s)\right)\, d\nu.\label{eq:gen}
\end{align}

Now we proceed with estimating the integrals
$$
I(\xi):=\int_{Y} \xi \left(\prod_{s=1}^r \varphi_s\circ a(t_s)\right)\, d\nu\quad\quad \hbox{with $\xi \in \Xi_Y.$}
$$
We claim that for some $\ell\ge  1$ and $\delta_r>0$,
\begin{align}\label{eq:claim}
I(1)&=\prod_{s=1}^r \int_{X_{I_s}} \varphi_s \,  d\mu_{I_s} + O\left(e^{\delta_r D(t_1,\ldots,t_k)}{\prod}_{s=1}^r \|\varphi_s\|_{C^\ell}\right),\nonumber\\
I(\xi) &= O\left(e^{\delta_r D(t_1,\ldots,t_k)}{\prod}_{s=1}^r \|\varphi_s\|_{C^\ell}\right),\quad \hbox{for $\xi\ne 1$},
\end{align}
uniformly on $\xi$.
Once this claim is established, Theorem \ref{th:multi_gen}
follows directly from \eqref{eq:gen}.
The proof proceeds by induction on $r$. \\

We use that the measure $\nu$ is invariant under the one-parameter subgroup
$\{\exp(uw)\}_{u\in\bR}$, so that for any $u\in\bR$,
\begin{align*}
I(\xi) =& \int_Y \xi\circ \exp(uw) \left( \prod_{s=1}^r \varphi_s\circ a(t_s)\circ \exp(uw)\right)\, d\nu\\
	=& \xi(\exp(uw))  \int_Y \xi \left(\prod_{s=1}^r \varphi_s\circ \exp(u w^{(s)})\circ a(t_s)\right)\, d\nu,
\end{align*}
where 
$$
w^{(s)}: = \Ad(a(t_s))w=e^{\alpha(t_s)}w.
$$
In particular, it follows from \eqref{eq:aaa} that
$$
\|w^{(1)}\|\ge \|w^{(2)}\|\ge\cdots \ge \|w^{(r)}\|.
$$
Integrating over $u\in [0,L]$, we deduce that for any $L>0$,
\begin{align}\label{eq:average}
	I(\xi) &= \int_{Y}  \xi \left(  \frac{1}{L}\int_0^L \xi(\exp(uw)) \prod_{s=1}^r \varphi_s\circ \exp(u w^{(s)})\circ a(t_s)\, du\right)\, d\nu.
\end{align}
We choose the parameter $L$ so that $L\|w^{(s)}\|$ is "large" for $s\le k$
 and "small" for $s\ge k+1$. The exact value of $L$ will be given at the end of the argument.

\medskip
\noindent{\sc Step I:} {\it  
Controlling the components	$\exp(u w^{(s)})$ with $s=k+1,\ldots, r$.
}

\vspace{0.1cm}

Let us consider a related integral
\begin{align}
	I_1(\xi) :=& 
	\int_{Y} \xi \left(\frac{1}{L}\int_0^L \xi(\exp(uw))\prod_{s=1}^k \varphi_s\circ \exp(uw^{(s)})\circ a(t_s)\, du\right) \left(   \prod_{s=k+1}^r   \varphi_s\circ a(t_s)\right)\, d\nu.
\end{align}
Since $|\xi|=1$, we obtain that
\begin{align}\label{i}
	|I(\xi)-I_{1}(\xi)| \leq &\prod_{s=1}^k \|\varphi_s\|_{C^0}  \\
	&\quad\quad\times \sup_{u \in [0,L]}\Big\|\prod_{s=k+1}^r \varphi_s \circ \exp (u w^{(s)}) \circ a(t_s) - \prod_{s=k+1}^r \varphi_s \circ a(t_s)\Big\|_{C^0}, \nonumber 
\end{align}
and by the triangle inequality,
\begin{align}\label{eq:almost}
	&\Big\|\prod_{s=k+1}^r \varphi_s \circ \exp(u w^{(s)}) \circ a(t_s) - \prod_{s=k+1}^r \varphi_s \circ a(t_s)\Big\|_{C^0}\\
	=&\,
	\Big\|\prod_{s=k+1}^r \varphi_s \circ \exp (u w^{(s)}) - \prod_{s=k+1}^r \varphi_s\Big\|_{C^0} \nonumber\\
	\le\; & \sum_{l=k+1}^{r}  
	\Big\|\Big(\prod_{s=k+1}^l \varphi_s \circ \exp (u w^{(s)})\Big) \prod_{s=l+1}^r \varphi_s - 
	\Big( \prod_{s=k+1}^{l-1} \varphi_s \circ \exp (u w^{(s)})\Big) \prod_{s=l}^r \varphi_s
	\Big\|_{C^0}\nonumber\\
	\le\; & \sum_{l=k+1}^{r}  \big\|\varphi_l \circ \exp (u w^{(l)}) - \varphi_l\big\|_{C^0} \prod_{k+1\le s\le r, s\ne \ell } \|\varphi_s\|_{C^0}\nonumber\\
	\ll&  L\|w^{(k+1)}\|\, {\prod}_{s=k+1}^r \|\varphi_s\|_{C^1}.\nonumber
\end{align}
Therefore, we conclude that
\begin{equation}\label{eq:step1_0}
|I(\xi)-I_{1}(\xi)| \ll L\|w^{(k+1)}\|\, {\prod}_{s=1}^r \|\varphi_s\|_{C^1}.
\end{equation}

\medskip
\noindent{\sc Step II:} {\it  	Applying the inductive hypothesis.}

\vspace{0.1cm}

We use notation: $I^{(1)}:=I\cup \{m+1,\ldots,m+n\}$ and $I^{(2)}:=I\cup \{1,\ldots, m\}$

When $w=E_{i_0j_0}$ is not mixing with respect to a set $I$,
it is convenient to introduce a unipotent subgroup $U_{I,w}$,
which is defined as follows:
\begin{itemize}
\item If $i_0\in I$  and $j_0\notin I$, 
the Lie algebra of $U_{I,w}$ is spanned by $E_{ij_0}$ with $i\in I^{(2)}$.
\item If $i_0\notin I$  and $j_0\in I$, 
the Lie algebra of $U_{I,w}$ is spanned by $E_{i_0 j}$ with $j\in I^{(1)}$.
\item If $i_0\notin I$  and $j_0\notin I$, 
the Lie algebra of $U_{I,w}$ is spanned by $E_{i_0 j_0}$.
\end{itemize}
We note that the spaces $X_I$ are isomorphic to the spaces from Remark \ref{rem:general}.
Indeed, in the first two cases, the action $\exp(\bR w)$ on $X_I$ corresponds
to one-parameter subgroups as in Remark \ref{rem:general}.
In the last case, since $U_{I,w}$ lies in the center of $G_I$,
the space $X_I$ is a bundle with fibers $gU_{I,w}\Gamma$ and 
the action on the fibers is given by 
$$
gz\mapsto g\exp(uw)z,\quad u\in\bR,\; z\in U_{I,w}\Gamma,
$$
so that we obtain the circle rotation which corresponds to the case $d=1$ in Section \ref{sec:decor}.
Hence, the results of Section \ref{sec:decor} are applicable in our setting.

We denote by $\nu_{I,w}$ the invariant probability measure supported on  $U_{I,w}\Gamma$.
When $w$ is not mixing for $I_s$, we set
$$
\varphi^\star_s(g\Gamma):=\int_{U_{I_s,w}\Gamma} \varphi_s(gy)\, d\nu_{I_s,w}(y).
$$
We note that $\varphi^\star_s$ is $\exp(\bR w)$-invariant,
and  $\int_{X_{I_s}} \varphi^\star_s\, d\mu_{I_s}=\int_{X_{I_s}} \varphi_s\, d\mu_{I_s}$.

When $w$ is mixing for $I_s$, we set
$$
\varphi^\star_s(g\Gamma):=\int_{X_{I_s}} \varphi_s\, d\mu_{I_s}.
$$

Let
\begin{align*}
	I_2(\xi) :=& 
	\int_{Y} \xi \left(\frac{1}{L}\int_0^L \xi(\exp(uw))\prod_{s=1}^k \varphi^\star_s\circ \exp(uw^{(s)})\circ a(t_s)\, du\right) \left(   \prod_{s=k+1}^r   \varphi_s\circ a(t_s)\right)\, d\nu\\
 \end{align*}
 and
\begin{align*}
I_3(\xi):=& 
	\int_{Y} \xi \left(\prod_{s=1}^k \varphi^\star_s\circ a(t_s)\right) \left(   \prod_{s=k+1}^r   \varphi_s\circ a(t_s)\right)\, d\nu.
\end{align*}
Arguing as in {\sc Step I}, we obtain that
\begin{equation}\label{eq:step1_0_0}
|I_2(\xi)-I_3(\xi)| \ll L\|w^{(k+1)}\|\, {\prod}_{s=1}^r \|\varphi_s\|_{C^1}.
\end{equation}
Since the functions $\varphi_s^\star$ are $\exp(\bR w)$-invariant,
$$
I_2(\xi) = 
	\int_{Y} \xi \left(\frac{1}{L}\int_0^L \xi(\exp(uw))\prod_{s=1}^k \varphi^\star_s\circ a(t_s)\, du\right) \left(   \prod_{s=k+1}^r   \varphi_s\circ a(t_s)\right)\, d\nu.
 $$

Let us introduce a function ${\Psi}_{\xi,L} : X^k \ra \bR$ defined by
\begin{align}
	\label{psitildeL}
	{\Psi}_{\xi,L}(\underline{x}):
	=&
	\frac{1}{L} \int_0^L \xi(\exp(uw))\Big( \prod_{s=1}^k \varphi_s(\exp(u w^{(s)})x_s) -  \prod_{s=1}^k \varphi^\star_s(\exp(u w^{(s)})x_s)\Big) \, du\\
 =&
	\frac{1}{L} \int_0^L \xi(\exp(uw))\Big( \prod_{s=1}^k \varphi_s(\exp(u w^{(s)})x_s) -  \prod_{s=1}^k \varphi^\star_s(x_s)\Big) \, du \nonumber
\end{align}
for $\underline{x}=(x_1,\ldots,x_k)\in X^k$.
Setting 
$$
\Phi: = \prod_{s=k+1}^r \varphi_s \circ a(t_s)\quad
\hbox{and}\quad a(t):=(a(t_1),\ldots,a(t_k)),
$$
we obtain that 
\[
I_{1}(\xi)-I_{2}(\xi) = \nu\big(\xi\cdot (\Psi_{\xi,L}\circ a(t))|_{diag(X^k)} \cdot \Phi\big).
\]
Hence, by Cauchy-Schwartz inequality,
\begin{align}\label{eq:step2_0}
|I_{1}(\xi)-I_{2}(\xi)| &\leq \nu\big((|\Psi_{\xi,L}|^2\circ a(t))|_{diag(X^k)}\big)^{1/2} \|\Phi\|_{C^0}\\
&\le \nu\big((|\Psi_{\xi,L}|^2\circ a(t))|_{diag(X^k)}\big)^{1/2}
	{\prod}_{s=k+1}^r \|\varphi_s\|_{C^0}. \nonumber
\end{align}
Let $\mu^{(k)}:=\prod_{s=1}^k\mu_{I_s}$ be the measure on $X^k$.
Using that the inequality 
$$
\alpha^{1/2} \leq |\alpha-\beta|^{1/2}+\beta^{1/2}\quad\hbox{ for all $\alpha,\beta \geq 0$,}
$$
we see that
\begin{align}\label{eq:step2_1}
	\nu\big((|\Psi_{\xi,L}|^2\circ a(t))|_{diag(X^k)}\big)^{1/2} 
	\leq&  \big|\nu\big((|\Psi_{\xi,L}|^2\circ a(t))|_{diag(X^k)}\big) - \mu^{( k)}\big(|{\Psi}_{\xi,L}|^2\big)\big|^{1/2} \\
	& \quad\quad + \mu^{(k)}\big(|{\Psi}_{\xi,L}|^2\big)^{1/2}. \nonumber
\end{align}

We proceed to estimate the first term in \eqref{eq:step2_1}.
The function $|\Psi_{\xi,L}|^2$ can be written as linear combination of functions of the form
$$
{\Omega}_{\xi,L}(\underline{x}):=\frac{1}{L^2}\int_0^L\int_0^L \xi((u_1-u_2)w) \left(\prod_{s=1}^k \rho_{1,s}(\exp(u_1w^{(s)})x_s)\rho_{2,s}(\exp(u_2w^{(s)})x_s)\right)\, du_1du_2,
$$
where $\rho_{i,s}$ is either $\varphi_s$, $\varphi^\star_s$, $\bar\varphi_s$, $\bar\varphi^\star_s$. 
Let
$$
\omega^{u_1,u_2}_s:= \big(\rho_{1,s}\circ\exp(u_1w^{(s)})\big)\cdot \big(\rho_{2,s}\circ\exp(u_2w^{(s)})\big).
$$
Then
\begin{align*}
\nu\big((\Omega_{\xi,L}\circ a(t))|_{diag(X^k)}\big)
=\frac{1}{L^2}\int_0^L\int_0^L \xi((u_1-u_2)w)\int_Y \left(\prod_{s=1}^k \omega_s^{u_1,u_2}\circ a(t_s)\right)\, d\nu\, du_1du_2.
\end{align*}
Since $k<r$, we can apply the inductive hypothesis to conclude that 
for some $\ell\ge 1$ and $\delta_k>0$,
$$
\int_Y \left(\prod_{s=1}^k \omega_s^{u_1,u_2}\circ a(t_s)\right)\, d\nu
=
 \prod_{s=1}^k \int_{X_{I_s}} \omega_s^{u_1,u_2}\,d\mu_{I_s}
 +O\left( e^{-\delta_k D(t_1,\ldots,t_k)}{\prod}_{s=1}^k \|\omega_s^{u_1,u_2}\|_{C^\ell} \right).
$$
Since 
\begin{align*}
	\mu^{(k)}\big(\Omega_{\xi,L}\big)
	=\frac{1}{L^2}\int_0^L\int_0^L \xi((u_1-u_2)w)\left(\prod_{s=1}^k \int_{X_{I_s}}  \omega_s^{u_1,u_2}\, d\mu_{I_s}\right)\, du_1du_2,
\end{align*}
it remains to estimate the error term. We have 
\begin{align*}
	\|\omega_s^{u_1,u_2}\|_{C^\ell}
	&\leq  \big\|(\rho_{1,s} \circ \exp(u_1 w^{(s)}))\cdot  (\rho_{2,s} \circ \exp(u_2 w^{(s)}))\big\|_{C^\ell} \\
	&\ll 
	\big\|\rho_{1,s} \circ \exp(u_1 w^{(s)})\big\|_{C^\ell}\cdot  \big\|\rho_{2,s} \circ \exp(u_2 w^{(s)})\big\|_{C^\ell}\\
	&\ll \max\big(1, L\|w^{(s)} \|\big)^{2\ell} \|\rho_{1,s}\|_{C^\ell}\|\rho_{2,s}\|_{C^\ell}\\
	&\le \max\big(1, L\|w^{(s)} \|\big)^{2\ell} \|\varphi_s\|_{C^\ell}^2.
\end{align*}
Hence, we deduce that
\begin{align*}
\Big|\nu\big((\Omega_{\xi,L}\circ a(t))|_{diag(X^k)}\big) - & \mu^{(k)}\big(\Omega_{\xi,L}\big)\Big|\\
\ll& \max\big(1,L\|w^{(1)}\|\big)^{2k\ell} e^{-\delta_k D(t_1,\ldots,t_k)} {\prod}_{s=1}^k \|\varphi_s\|_{C^\ell}^2.
\end{align*}
Recalling that $|\Psi_{\xi,L}|^2$ is a linear combination of functions 
of the form $\Omega_{\xi,L}$, we also conclude that
\begin{align}
	\label{eq:step2_2}
	\Big|\nu\big((|\Psi_{\xi,L}|^2\circ a(t))|_{diag(X^k)}\big) -& \mu^{(k)}\big(|\Psi_{\xi,L}|^2\big)\Big|\\
	&\ll \max\big(1,L\|w^{(1)}\|\big)^{2k\ell } e^{-\delta_k D(t_1,\ldots,t_k)} {\prod}_{s=1}^k \|\varphi_s\|_{C^\ell}^2. \nonumber
\end{align}

\medskip
\noindent{\sc Step III:} {\it  
	Decorrelation estimates for the flows $\exp(u w^{(s)})$ with $s=1,\ldots, k$.
}

\vspace{0.1cm}

Now we proceed with estimating the second term in \eqref{eq:step2_1}.
We write ${\Psi}_{\xi,L}$ as
\[
{\Psi}_{\xi,L}(\underline{x}) = \frac{1}{L} \int_0^L \xi(\exp(uw)) \Theta_u(\underline{x}) \, du
\]
with
\[
\Theta_u(\underline{x}) := \prod_{s=1}^k \varphi_s\big(\exp(u w^{(s)})x_s\big) - \prod_{s=1}^k \varphi^\star_s(x_s),
\]
so that 
$$
\left|{\Psi}_{\xi,L}(\underline{x})\right|^2=\frac{1}{L^2} \int_0^L\int_0^L \xi(\exp((u_1-u_2)w)) \Theta_{u_1}(\underline{x})\overline{\Theta}_{u_2}(\underline{x}) \, du_1du_2.
$$
We have 
\begin{align*}
\Theta_{u_1}(\underline{x})\overline{\Theta}_{u_2}(\underline{x})
=&\prod_{s=1}^k \varphi_s\big(\exp(u_1 w^{(s)})x_s\big)\bar\varphi_s\big(\exp(u_2 w^{(s)})x_s\big)\\
&\quad\quad-\prod_{s=1}^k \varphi_s\big(\exp(u_1 w^{(s)})x_s\big)\bar\varphi^\star_s(x_s) 
 -\prod_{s=1}^k \varphi^\star_s(x_s)  \bar \varphi_s\big(\exp(u_2 w^{(s)})x_s\big) \\
&\quad\quad+\prod_{s=1}^k \varphi_s^\star(x_s)\bar\varphi^\star_s(x_s).
\end{align*}
Since 
$$
\int_{X_{I_s}} \big(\varphi_s\circ \exp(u_1 w^{(s)})\big)\cdot \bar\varphi^\star_s\, d\mu_{I_s}=
\int_{X_{I_s}} \varphi^\star_s\cdot\big(  \bar \varphi_s\circ \exp(u_2 w^{(s)})\big)\, d\mu_{I_s}=
\int_{X_{I_s}} \varphi^\star_s \bar\varphi^\star_s\, d\mu_{I_s},
$$
we obtain that
\begin{align*}
\mu^{(k)}\left( \Theta_{u_1}\cdot \overline{\Theta}_{u_2}\right)
=\prod_{s=1}^{k} \mu_{I_s}\big(\varphi_s\circ \exp( u_1 w^{(s)}) \cdot {\bar\varphi_s\circ \exp( u_2 w^{(s)})} \big)-
\prod_{s=1}^{k} \mu_{I_s}\big(\varphi^\star_s\cdot \bar\varphi^\star_s\big).
\end{align*}
We claim that
\begin{align}\label{eq:step2_3}
\mu^{(k)}\big(|{\Psi}_{\xi,L}|^2\big) &=\frac{1}{L^2} \int_0^L\int_0^L \xi(\exp((u_1-u_2)w)) \mu^{(k)}\left(\Theta_{u_1}\cdot\overline{\Theta}_{u_2}\right) \, du_1du_2\\
&\ll
\max\big(1,L\|w^{(k)}\|\big)^{-c}\, {\prod}_{s=1}^k \|\varphi_s\|_{C^\ell}^2\nonumber
\end{align}
for some $c>0$. 
We write 
$$
\{1\ldots,k\}=S_{mix}\sqcup S_{mix}^c,
$$ where $S_{mix}$ consists of such indices $s$ that $w$ is a mixing direction for $I_s$.
By Corollary~\ref{cor:decor}, for $s\in S_{mix}$,
\begin{align}\label{eq:mix_last}
\mu_{I_s}\big(\varphi_s\circ \exp( u_1 w^{(s)}) \cdot {\bar\varphi_s\circ \exp( u_2 w^{(s)})} \big)=&
\mu_{I_s}\big(\varphi^\star_s\cdot \bar\varphi^\star_s\big)\\
&\;+O\Big(\max\big(1,|u_1-u_2|\|w^{(s)}\|\big)^{-\tau}\, \|\varphi_s\|_{C^\ell}^2\Big). \nonumber
\end{align}
Of course, in this case,  $\mu_{I_s}\big(\varphi^\star_s\cdot \bar\varphi^\star_s\big)=\mu_{I_s}(\varphi_s)\mu_{I_s}(\bar\varphi_s)$.

Let 
\[
\Theta^\star_u(\underline{x}) := \prod_{s\in S_{mix}^c} \varphi_s\big(\exp(u w^{(s)})x_s\big)\cdot\prod_{s\in S_{mix}} \varphi^\star_s(x_s) - \prod_{s=1}^k \varphi^\star_s(x_s)
\]
and 
\[
{\Psi}^\star_{\xi,L}(\underline{x}) := \frac{1}{L} \int_0^L \xi(\exp(uw)) \Theta^\star_u(\underline{x}) \, du.
\]
As above, we obtain that 
$$
\mu^{(k)}\big(|{\Psi}^\star_{\xi,L}|^2\big) =\frac{1}{L^2} \int_0^L\int_0^L \xi(\exp((u_1-u_2)w)) \mu^{(k)}\left(\Theta^\star_{u_1}\cdot\overline{\Theta}^\star_{u_2}\right) \, du_1du_2,
$$
and 
\begin{align*}
\mu^{(k)}\left( \Theta^\star_{u_1}\cdot \overline{\Theta}^\star_{u_2}\right)
=&\prod_{s\in S_{mix}^c} \mu_{I_s}\big(\varphi_s\circ \exp( u_1 w^{(s)}) \cdot {\bar\varphi_s\circ \exp( u_2 w^{(s)})} \big)\prod_{s\in S_{mix}} \mu_{I_s}\big(\varphi^\star_s\cdot \bar\varphi^\star_s\big)\\
&\quad\quad-\prod_{s=1}^{k} \mu_{I_s}\big(\varphi^\star_s\cdot \bar\varphi^\star_s\big).
\end{align*}
It follows from \eqref{eq:mix_last} that 
\begin{align*}
\mu^{(k)}\left( \Theta_{u_1}\cdot \overline{\Theta}_{u_2}\right)=&\mu^{(k)}\left( \Theta^\star_{u_1}\cdot \overline{\Theta}^\star_{u_2}\right)\\
& \quad+O\Big(\max\big(1,|u_1-u_2|\|w^{(k)}\|\big)^{-\tau}\, {\prod}_{s=1}^k \|\varphi_s\|_{C^\ell}^2\Big),
\end{align*}
so that using \cite[Lemma~4.3]{BG2}, we conclude that 
\begin{align*}
\Big| \mu^{(k)}\big(|{\Psi}_{\xi,L}|^2\big)-\mu^{(k)}\big(|{\Psi}^\star_{\xi,L}|^2\big)\Big| 
&\le\frac{1}{L^2} \int_0^L\int_0^L \big| \mu^{(k)}\left(\Theta_{u_1}\cdot\overline{\Theta}_{u_2}\right)-
\mu^{(k)}\left(\Theta^\star_{u_1}\cdot\overline{\Theta}^\star_{u_2}\right)\big|\, du_1du_2\\
&\ll
\max\big(1,L\|w^{(k)}\|\big)^{-\tau}\, {\prod}_{s=1}^k \|\varphi_s\|_{C^\ell}^2.\nonumber
\end{align*}
Finally, the estimate for $\mu^{(k)}\big(|{\Psi}^\star_{\xi,L}|^2\big)$ reduces to analysis of the integrals
\begin{align*}
\frac{1}{L^2} \int_0^L\int_0^L \xi(\exp((u_1-u_2)w)) \Big(\prod_{s\in S_{mix}^c} \mu_{I_s}\big(\varphi_s\circ   \exp( u_1  & w^{(s)})  \cdot {\bar\varphi_s\circ \exp( u_2 w^{(s)})} \big)\\
&-\prod_{s\in S_{mix}^c} \mu_{I_s}\big(\varphi^\star_s\cdot \bar\varphi^\star_s\big)\Big)\, du_1du_2.
\end{align*}
In this case, according to Corollary \ref{cor:eq_part} and Remark \ref{rem:general}, we also have the bound
$$
\ll \max\big(1,L\|w^{(k)}\|\big)^{-\tau'}\, {\prod}_{s=1}^k \|\varphi_s\|_{C^\ell}^2
$$
for some $\tau'>0$. This verifies \eqref{eq:step2_3}.

\medskip

Combining \eqref{eq:step2_0}, \eqref{eq:step2_1}, \eqref{eq:step2_2}, and 
\eqref{eq:step2_3}, we conclude that 
\begin{align}
	\label{eq:step1_1}
|I_{1}(\xi)-I_{2}(\xi)|\ll \Big(
\max\big(1,L\|w^{(1)}\|\big)^{k\ell} &e^{-\delta_k D(t_1,\ldots,t_k)/2}\\
&+ \max\big(1,L\|w^{(k)}\|\big)^{-c/2}\Big){\prod}_{s=1}^r \|\varphi_s\|_{C^\ell}. \nonumber
\end{align}

\medskip
\noindent{\sc Step IV:} {\it  
	Optimizing the parameters.
}

\vspace{0.1cm}

Combining the bound \eqref{eq:step1_1} with \eqref{eq:step1_0}  and \eqref{eq:step1_0_0}, we derive that
\begin{align*}
\int_{Y} \xi \left(\prod_{s=1}^r \varphi_s\circ a(t_s)\right)\, d\nu=&	
\int_{Y} \xi\left(\prod_{s=1}^k \varphi^\star_s\circ a(t_s)\right) \left(   \prod_{s=k+1}^r   \varphi_s\circ a(t_s)\right)\, d\nu\\
&+O\Big(L\|w^{(k+1)}\|+\max\big(1,L\|w^{(1)}\|\big)^{k\ell } e^{-\delta_k D(t_1,\ldots,t_k)/2}\\
&\quad\quad\quad\quad\quad+ \max\big(1,L\|w^{(k)}\|\big)^{-c/2}\Big){\prod}_{s=1}^r \|\varphi_s\|_{C^\ell}. 
\end{align*}
Since $\|w^{(s)}\|=e^{\alpha(t_s)}$, 
it remains to optimize
$$
L\,e^{\alpha(t_{k+1})}+\max\big(1,L\,e^{\alpha(t_1)}\big)^{k\ell } \cdot e^{-\delta_k D(t_1,\ldots,t_k)/2}
+ \max \big(1,L\, e^{\alpha(t_k)}\big)^{-c/2}
$$
by choosing $L$ suitably. We recall that
$$
\alpha(t_{k}-t_{k+1})\ge \theta/r \, D(t_1,\ldots.t_r)
$$
We consider two cases:

\vspace{0.2cm}

\noindent{\sc Case A:} {\it  
	$\alpha(t_1-t_k)\le \rho\, D(t_1,\ldots,t_r)$ with $\rho:=\delta_k/(8k\ell )$.
}

\vspace{0.1cm}

In this case, we take 
$$
L=e^{-\alpha(t_k)}\cdot e^{\epsilon D(t_1,\ldots,t_r)}\quad\hbox{with $\epsilon=\min(\theta/(2r),\rho)$.}
$$
Then 
\begin{align*}
L\,e^{\alpha(t_{k+1})}&=e^{\alpha(t_{k+1}-t_k)}\cdot e^{\epsilon D(t_1,\ldots,t_r)}\le e^{-\epsilon\, D(t_1,\ldots,t_r)},\\
\max\big(1,L\,e^{\alpha(t_1)}\big)^{k\ell} e^{-\delta_k D(t_1,\ldots,t_k)/2}&\le 
e^{-(\delta_k/2-(\rho+\eps)\cdot k\ell) D(t_1,\ldots,t_k)}
 = e^{-\delta_k/4\, D(t_1,\ldots,t_r)},\\
\max \big(1,L\, e^{\alpha(t_k)}\big)^{-c/2} &\le e^{-\epsilon c/2\, D(t_1,\ldots,t_r)}.
\end{align*}
Therefore, we conclude that for some $\delta'>0$,
\begin{align*}
	\int_{Y} \xi \left(\prod_{s=1}^r \varphi_s\circ a(t_s)\right)\, d\nu=&	
	\int_{Y} \xi\left(\prod_{s=1}^k \varphi^\star_s\circ a(t_s)   \prod_{s=k+1}^r   \varphi_s\circ a(t_s)\right)\, d\nu\\
	&\quad\quad\quad\quad\quad\quad +O\Big( e^{-\delta'\, D(t_1,\ldots,t_r)}\Big){\prod}_{s=1}^r \|\varphi_s\|_{C^\ell}. 
\end{align*}
Furthermore, for some $s_0=1,\ldots k$, the direction $w$ is mixing for $I_{s_0}$, so that 
$$
\varphi^\star_{s_0}=\int_{X_{I_{s_0}}} \varphi_{s_0}\, d\mu_{I_{s_0}}.
$$
Then 
\begin{align}\label{eq:llast}
&\int_{Y}\xi \left(\prod_{s=1}^k \varphi^\star_s\circ a(t_s)\right) \left(   \prod_{s=k+1}^r   \varphi_s\circ a(t_s)\right)\, d\nu\\
=&
\int_{X_{I_{s_0}}} \varphi_{s_0}\, d\mu_{I_{s_0}} \cdot
\int_{Y} \xi\left(\prod_{1\le s\le k, s\ne s_0} \varphi^\star_s\circ a(t_s)\right) \left(   \prod_{s=k+1}^r   \varphi_s\circ a(t_s)\right)\, d\nu,\nonumber
\end{align}
and the estimate follows by induction on the number of factors.

\vspace{0.2cm}

\noindent{\sc Case B:} {\it  
	$\alpha(t_1-t_k)> \rho\, D(t_1,\ldots,t_r)$ with $\rho=\delta_k/(8k\ell)$.
}

\vspace{0.1cm}

In this case, there exists $k_1=1,\ldots k-1$ such that 
\begin{equation}
	\label{eq:gapp}
\alpha(t_{k_1}-t_{k_1+1})> \rho/(k-1)\cdot D(t_1,\ldots,t_r)
\end{equation}
Then we change our choice of $k$ to $k_1$ and apply the above argument with $k_1$.
Again we land either in {\sc Case A} or {\sc Case B} (with $k$ replaced by $k_1$).
In {\sc Case A},  we get
\begin{align}\label{eq:stepp}
	\int_{Y} \xi\left(\prod_{s=1}^r \varphi_s\circ a(t_s)\right)\, d\nu=&	
	\int_{Y} \xi \left(\prod_{s=1}^{k_1} \varphi^\star_s\circ a(t_s) \right) \left(  \prod_{s=k_1+1}^r   \varphi_s\circ a(t_s)\right)\, d\nu\\
	&\quad+O\Big( e^{-\delta'\, D(t_1,\ldots,t_r)}\Big){\prod}_{s=1}^r \|\phi_s\|_{C^\ell} \nonumber
\end{align}
for some $\delta'>0$.
In {\sc Case B}, we can replace $k_1$ by a smaller index, so that 
a version of \eqref{eq:gapp} still holds.
Ultimately, for some choice of $k_1\le k$ 
we get the estimate \eqref{eq:stepp},
and it remains to analyze the latter integral.
If there exists $s\le k_1$ such that $w$ is a mixing direction for $I_s$,
we can argue as in \eqref{eq:llast} to finish the proof.
Otherwise, we may assume that the functions $\varphi_s$, $s=1,\ldots,k_1$,
are invariant under the subgroup $U_{I_s,w}$.
Then \eqref{eq:average} can be rewritten as 
\begin{align*}
	I(\xi) &= \int_{Y}\xi \left(  \frac{1}{L}\int_0^L \xi(\exp(uw)) \prod_{s=1}^{k_1}
	\varphi_s\circ a(t_s) \prod_{s=k_1+1}^r \varphi_s\circ \exp(u w^{(s)})\circ a(t_s)\, du\right)\, d\nu.
\end{align*}
We note that the one-parameter subgroups $\exp(u w^{(s)})$ appear only in the factors $s=k_1+1,\ldots,r$. In this case, our argument gives
\begin{align*}
	&\int_{Y} \xi  \left(\prod_{s=1}^r \varphi_s\circ a(t_s)\right)\, d\nu\\
	=&	
	\int_{Y} \xi \left(\prod_{s=1}^k \varphi^\star_s\circ a(t_s)\right) \left(   \prod_{s=k+1}^r   \varphi_s\circ a(t_s)\right)\, d\nu\\
	&\quad\quad+O\Big(L\|w^{(k+1)}\|
	+\max\big(1,L\|w^{(k_1+1)}\|\big)^{(k-k_1)\ell} e^{-\delta_{k-k_1}D(t_{k_1+1},\ldots,t_k)/2}\\
	&\quad\quad\quad\quad\quad\quad\quad\quad+ \max\big(1,L\|w^{(k)}\|\big)^{-c/2}\Big){\prod}_{s=1}^r \|\varphi_s\|_{C^\ell}. 
\end{align*}
As above we consider two cases:
\begin{itemize}
	\item {\sc Case A:} $\alpha(t_{k_1+1}-t_k)\le \rho\, D(t_1,\ldots,t_r)$ with $\rho:=\delta_{k-k_1}/(8(k-k_1)\ell)$.
	\item {\sc Case B:} $\alpha(t_{k_1+1}-t_k)> \rho\, D(t_1,\ldots,t_r)$.
\end{itemize}

In {\sc Case A}, we complete the proof by induction on the number of factors
using that 
$\varphi^\star_{s_0}=\int_{X_{I_{s_0}}} \varphi_{s_0}\, d\mu_{I_{s_0}}$
for some $s_0=k_1+1,\ldots k$.

In {\sc Case B}, we argue as above to show that there exists $k_2=k_1+1,\ldots,k$
such that  
\begin{align*}
	\int_{Y} \xi \left(\prod_{s=1}^r \varphi_s\circ a(t_s)\right)\, d\nu=&	
	\int_{Y} \xi \left(\prod_{s=1}^{k_2} \varphi^\star_s\circ a(t_s)  \right) \left(\prod_{s=k_2+1}^r   \varphi_s\circ a(t_s)\right)\, d\nu\\
	&\quad\quad\quad\quad+O\Big( e^{-\delta'\, D(t_1,\ldots,t_r)}\Big){\prod}_{s=1}^r \|\varphi_s\|_{C^\ell}
\end{align*}
for some $\delta'>0$. Therefore, we can assume that 
all functions $\varphi_s$ with $s=1,\ldots, k_2$  are are $U_{w,I_s}$-invariant. 

Arguing this way, we eventually either land in {\sc Case A} or
reach the point when $s_0\le k_2$, so that we can complete 
the proof by induction on the number of factors.

This completes the proof of Theorem \ref{th:multi_gen}.

\section{Quasi-independence for separated tuples}\label{sec:separated}

For $I \subset \{1,\ldots,m+n\}$ and $t\in A^+$, we define 
\[
\lfloor t \rfloor_I := \min_{i \in I} t_i \qand \lceil t \rceil_I := \max_{i \in I} t_i.
\]
When $I=\emptyset$, we set $\lfloor t \rfloor_I= \lceil t \rceil_I =0$.
For simplicity, we also write
\[
\lfloor t \rfloor 
:= \lfloor t \rfloor_{\{1,\ldots,m+n\}} \qand \lceil t \rceil := \lceil t \rceil_{\{1,\ldots,m+n\}}.
\]

We fix $\ell\ge 1$ and $\delta>0$ such that Theorem  \ref{th:multi_gen} holds.

\begin{lemma}
\label{l:induction}
Let $\alpha,  \beta\in (0,1)$ and let $I_1,\ldots,  I_r \subseteq \{1,\ldots,m+n\}$ be admissible subsets. 
Then
for every $f\in C^\infty(Y)$, $\varphi_1,\ldots,\varphi_r\in C^\infty_c(X)$, and 
 $t_1,\ldots,  t_r \in A^{+}$ such that
\begin{align}
\min\big(\lfloor t_1\rfloor_{I_1},\ldots,  \lfloor t_r\rfloor_{I_r}\big)  &\geq \alpha\, \Delta(t_1,\ldots, t_r), \label{eq:1}\\
\max\big( \lceil t_1 \rceil_{I_1^c},\ldots, \lceil t_r \rceil_{I_r^c}\big) &\leq \beta\, \Delta(t_1,\ldots, t_r), \label{eq:2}
\end{align}
we have
\begin{align*}
\int_Y f \left(\prod_{s=1}^r\varphi_s \circ a(t_s)  \right)\,  d\nu 
= &\int_Y f\, d\nu\prod_{s=1}^r\int_{Y} \varphi_s \circ a(t_s) \,  d\nu \nonumber \\
&\quad\quad\quad+O\left(
e^{-(\delta\alpha-2r\ell \beta) \Delta(t_1,\ldots,t_r)} \,  \|f\|_W {\prod}_{s=1}^r \|\varphi_s\|_{C^\ell}\right).
\end{align*}
\end{lemma}

\begin{proof}
We write
\[
t_s = t'_{s} + t''_{s}\quad\hbox{with $t'_{s} \in A^+_{I_s}$ and $t''_{s} \in A^+_{I_s^c}$.}
\]
Then
\[
\int_Y f\left(\prod_{s=1}^r\varphi_s \circ a(t_s)  \right)\,  d\nu 
=
\int_Y f\left(\prod_{s=1}^r\psi_s \circ a(t'_s)  \right)\,  d\nu,
\]
where
\[
\psi_s: = \varphi_s \circ a(t''_{s}).
\]
By Theorem \ref{th:multi_gen}, 
\begin{align}
\int_Y f\left(\prod_{s=1}^r \psi_s \circ a(t'_{s})\right) \,  d\nu
=& \int_Y f\, d\nu \prod_{s=1}^r \int_{X_{I_s}} \psi_s \,  d\mu_{I_s} \label{eq:prod}
\\[0.2cm]
&\quad+O\left(
e^{-\delta \min\big(\lfloor t'_{1}\rfloor_{I_1},\ldots,\lfloor t'_{r} \rfloor_{I_r},\Delta(t'_1,\ldots,t'_r)\big)} \, \|f\|_W{\prod}_{s=1}^r \|\psi_s\|_{C^\ell} \right).\nonumber
 \end{align}

We claim that for any $1\le s_1,s_2\le r$,
\begin{equation}\label{eq:ineq}
\|t'_{s_1}-t'_{s_2}\|\ge \|t_{s_1}-t_{s_2}\|.
\end{equation}
Let us suppose that $\|t_{s_1}-t_{s_2}\|=|t_{s_1,k}-t_{s_2,k}|$
for some $1\le k\le m+n$.
Without loss of generally, we may assume that $t_{s_1,k}\ge t_{s_2,k}$. 
When $k\in I_{s_1}\cap I_{s_2}$,
$$
|t_{s_1,k}-t_{s_2,k}|=|t'_{s_1,k}-t'_{s_2,k}|\le \|t'_{s_1}-t'_{s_2}\|,
$$
so that \eqref{eq:ineq} holds. 
When $k\in I_{s_1}\backslash I_{s_2}$, we also obtain that
$$
|t_{s_1,k}-t_{s_2,k}|\le t_{s_1,k}=|t'_{s_1,k}-t'_{s_2,k}|
$$
since $t'_{s_2,k}=0$.
When $k\notin I_{s_1}$, it follows from \eqref{eq:2} that 
$$
|t_{s_1,k}-t_{s_2,k}|\le t_{s_1,k}\le \beta\, \|t_{s_1}-t_{s_2}\|.
$$
Since $\beta<1$, this case is not possible. This verifies \eqref{eq:ineq}.

In view of  \eqref{eq:ineq} and \eqref{eq:1}, the estimate \eqref{eq:prod} gives
\begin{align*}
&\int_Y f \left(\prod_{s=1}^r \psi_s \circ a(t'_{s})\right) \,  d\nu\\
=& \int_Y f\, d\nu\prod_{s=1}^r \int_{X_{I_s}} \psi_s \,  d\mu_{I_s} 
\\[0.2cm]
&\quad\quad+\left(
e^{-\delta \min\big(\lfloor t'_{1}\rfloor_{I_1},\ldots,\lfloor t'_{r} \rfloor_{I_r},\Delta(t_1,\ldots,t_r)\big)} \, \|f\|_W {\prod}_{s=1}^r \|\psi_s\|_{C^\ell} \right)\\[0.2cm]
=& \int_Y f\, d\nu\prod_{s=1}^r \int_{X_{I_s}} \psi_s \,  d\mu_{I_s} 
+O\left(
e^{-\delta\alpha\, \Delta(t_1,\ldots,t_r)} \, \|f\|_W{\prod}_{s=1}^r \|\psi_s\|_{C^\ell} \right).
\end{align*}
Since by Theorem \ref{th:multi_gen} with $r=1$,
\[
\int_{Y} \psi_s \circ a(t'_{s}) \,  d\nu 
= \int_{X_{I_s}} \psi_s \,  d\mu_{I_s} + O\left(e^{-\delta \lfloor t'_{s}\rfloor_{I_s} }\|\psi_s\|_{C^\ell}\right),
\]
we conclude that 
$$
\prod_{s=1}^r \int_{X_{I_s}} \psi_s \,  d\mu_{I_s} 
=\prod_{s=1}^r \int_{Y} \psi_s \circ a(t'_{s}) \,  d\nu 
+O\left(
e^{-\delta \min\big(\lfloor t'_{1}\rfloor_{I_1},\ldots,\lfloor t'_{r} \rfloor_{I_r}\big)} \, {\prod}_{s=1}^r \|\psi_s\|_{C^\ell} \right),
$$
and
\begin{align}\label{eq:lll}
\int_Y f\left(\prod_{s=1}^r \psi_s \circ a(t'_{s})\right) \,  d\nu
=& \int_Y f\, d\nu \prod_{s=1}^r \int_{Y} \psi_s \circ a(t'_{s}) \,  d\nu \\
&\quad+O\left(
e^{-\delta \alpha\, \Delta(t_1,\ldots,t_r)} \, \|f\|_W {\prod}_{s=1}^r \|\psi_s\|_{C^\ell} \right).\nonumber 
\end{align}
where
$$
\prod_{s=1}^r \int_{Y} \psi_s \circ a(t'_{s}) \,  d\nu
=  \prod_{s=1}^r \int_{Y} \varphi_s \circ a(t_{s}) \,  d\nu.
$$
Furthermore, it follows from \eqref{eq:2} that 
\[
\|\psi_s\|_{C^\ell} \ll e^{2\ell \lceil t''_{s} \rceil_{I_s^c}} \|\varphi_s\|_{C^\ell} = e^{2\ell \lceil t_{s} \rceil_{I_s^c}} \|\varphi_s\|_{C^\ell} \leq e^{2\ell \beta\, \Delta(t_1,\ldots,t_r)}\|\varphi_s\|_{C^\ell}.
\]
This implies the result.
\end{proof}

\section{Proof of main theorem} \label{sec:proof}

We fix $\ell\ge 1$ and $\delta>0$ such that Theorem  \ref{th:multi_gen} holds and set
\begin{equation}
	\label{def_lambda}
	\lambda := \max\big(2, {4r\ell}/{\delta}\big). 
\end{equation}

Since the case $r=1$ is trivial, 
we may assume by induction that the theorem holds for a smaller number of factors. Namely, for every $r_1\le r-1$, there exists $\eta(r_1)>0$
\begin{align}\label{eq:induction}
	\int_Y \left(\prod_{s=1}^{r_1} \varphi_s \circ a(t_s) \right) \,  d\nu =&
	\prod_{s=1}^{r_1} \int_{Y} \varphi_s \circ a(t_s) \,  d\nu 
	+ O\left( e^{-\eta(r_1)\, \Delta(t_1,\ldots,t_r)} {\prod}_{s=1}^{r_1} \|\varphi_s\|_{C^\ell}\right).
\end{align}

Throughout the proof, we use constants $c_1,\ldots,c_r$ such that 
\begin{equation}\label{eq:c1}
c_r\ge c_{r-1}\ge \cdots \ge  c_1 > \max(m,n),
\end{equation}
which will be specified later.

We split the set of parameters $(t_1,\ldots,t_r) \in (A^+)^r$
into multiple domains defined in terms of the constants $c_1,\ldots,c_r$.
These domains correspond to {\sc Cases $1^\prime_k$,
$1^{\prime\prime}_k$, $2_k$} with $k=1,\ldots,r$, described below.

\vspace{0.2cm}

\noindent{\sc Case $1_r$:} {\it there exist admissible subsets $I_1,\ldots,I_{r} \subset \{1,\ldots,m+n\}$ such that
\[
\Delta(t_1,\ldots,t_r) \leq c_r \min\big(\lfloor{t_1}\rfloor_{I_1}, \ldots, \lfloor{t_r}\rfloor_{I_r}\big).
\]
}

\vspace{0.1cm}

In this case, we distinguish two possibilities --- {\sc Cases  $1^\prime_r$} and
$1^{\prime\prime}_r$:

\vspace{0.2cm}
\noindent{\sc Case $1^\prime_r$:} {\it  the lower bound
\[
\Delta(t_1,\ldots,t_r)\geq \lambda c_r \max\big(\lceil t_1 \rceil_{I_1^c},\ldots, \lceil t_r \rceil_{I_r^c}\big)
\]
holds, 
where $\lambda$ is defined in \eqref{def_lambda}.  
}

\vspace{0.1cm}

In this case,  Lemma \ref{l:induction} is applicable with $\alpha = c_r^{-1}$ and $\beta = (\lambda c_r)^{-1} < c_r^{-1}$,  and we conclude that
\begin{align*}
\int_Y \left(\prod_{s=1}^{r} \varphi_s \circ a(t_s)\right) \,  d\nu 
&=  \prod_{s=1}^{r} \int_{Y} \varphi_s \circ a(t_s) \,  d\nu
+O\left(
e^{-(\delta/c_r - 2r\ell/\lambda c_r)\Delta(t_1,\ldots,t_r)} \, {\prod}_{s=1}^r \|\varphi_s\|_{C^\ell}\right)\\
&= \prod_{s=1}^r \int_{Y} \varphi_s \circ a(t_s) \,  d\nu
+O\left(
e^{-\delta/(2c_r)\, \Delta(t_1,\ldots,t_r)} \, {\prod}_{s=1}^r \|\varphi_s\|_{C^\ell} \right).
\end{align*}
This gives the required estimate. 

On the other hand, if the assumption in {\sc Case $1^\prime_r$} fails,
we have:

\vspace{0.2cm}

\noindent{\textsc{Case $1^{\prime\prime}_r$}:} {\it at least one of the sets
\[
J_s: = \big\{ j \in I_s^c \,  : \,  t_{s,j} > (\lambda c_r)^{-1}\Delta(t_1,\ldots,t_r) \big\},\quad s=1,\ldots,r,
\]
is non-empty.}

\vspace{0.1cm}

In this case, we introduce the new sets
\[
I'_{s} := I_s \sqcup J_s,
\]
and note that
\[
\Delta(t_1,\ldots,t_r) \leq \lambda c_r \min\big(\lfloor t_1 \rfloor_{I_1'},\ldots,\lfloor t_r \rfloor_{I_r'}\big).
\]
This means that we are in {\sc case $1_r$}
with $\lambda c_r$ instead of $c_r$ and $I_s'$
instead of $I_s$.  The key point in this argument is that at least one of  the sets $I'_s$ is strictly larger than the corresponding original sets $I_s$.

The new \textsc{Case $1^\prime_r$} is that
\[
\lambda^2 c_r \max\big(\lceil t_1 \rceil_{(I'_1)^c},\ldots \lceil t_r \rceil_{(I'_r)^c}\big) \leq \Delta(t_1,\ldots,t_r),
\]
which then implies as before that
\begin{align*}
\int_Y \left(\prod_{s=1}^r \varphi_s \circ a(t_s)\right) \,  d\nu 
=& \prod_{s=1}^r \int_{Y} \varphi_s \circ a(t_s) \,  d\nu
+O\left(
e^{-\delta/(2\lambda c_r)\, \Delta(t_1,\ldots,t_r)} \, {\prod}_{s=1}^r \|\varphi_s\|_{C^\ell} \right).
\end{align*}
The new \textsc{Case $1^{\prime\prime}_r$} is that at least one of the sets
\[
J'_s := \big\{ j \in (I'_s)^c \,  : \,  t_{s,j} > (1/\lambda^2 c_r)\Delta(t_1,\ldots,t_r) \big\}, \quad s=1,\ldots,r,
\]
is non-empty.  If this happens,  we consider again the new sets
\[
I_s'' := I'_s \sqcup J'_s
\]
and note that
\[
\Delta(t_1,\ldots,t_r) \leq \lambda^2 c_r \min\big(\lfloor t_1 \rfloor_{I_1''},\ldots,\lfloor t_r \rfloor_{I_r''}\big).
\]
We can now again restart {\sc Case $1_r$},  but with $\lambda^2 c_r$ instead of $\lambda c_r$,
and $I_s''$  instead of $I_s'$.  Again at least one of the sets $I_s''$ has strictly increased.  
For each such restart,  we are either in the favourable \textsc{Case $1_r^\prime$},  when we always get a suitable bound,  or we are in
the complementary \textsc{Case $1_r''$},  which allows us to increase (at least one of) the index sets,  and give \textsc{Case $1_r^\prime$} a new chance.  Since there are at most $r(m+n)$ indices in the union of the $r$ index sets,  there can be at most $r(m+n)$ restarts,
after which we must reach \textsc{Case $1_r^\prime$}.  
Therefore, the resulting bound after these iterations is: 
\begin{align*}
\int_Y \left(\prod_{s=1}^r \varphi_s \circ a(t_s)\right) \,  d\nu 
=& \prod_{s=1}^r \int_{Y} \varphi_s \circ a(t_s) \,  d\nu\\
&\quad\quad +O\left(
e^{-\delta/(2\lambda^{r(m+n)} c_r)\, \Delta(t_1,\ldots,t_r)} \, {\prod}_{s=1}^r \|\varphi_s\|_{C^\ell} \right).
\end{align*}
Hence, in {\sc Case $1_r$} we always get a favourable bound.

The complement of {\sc Case $1_r$}  is:

\vspace{0.2cm}

\noindent {\sc Case $2_r$:} {\it for {all} admissible $I_1,\ldots,  I_r \subset \{1,\ldots,m+n\}$,  we have }
\[
\Delta(t_1,\ldots,t_r) > c_r \min\big(\lfloor t_1\rfloor_{I_1},\ldots  \lfloor t_r\rfloor_{I_r}\big).
\]

\vspace{0.1cm}
For $t\in A^+$, we set
$$
\lceil t\rceil^\prime:=\min\big(\lceil t \rceil_{\{1,\ldots,m\}}, \lceil t \rceil_{\{m+1,\ldots,m+n\}} \big).
$$
We recall that for $t\in A^+$, one has $t_1+\cdots+t_m=t_{m+1}+\cdots+t_{m+n}$,
so that
$$ 
\lceil t \rceil_{\{1,\ldots,m\}}\le n\cdot \lceil t \rceil_{\{m+1,\ldots,m+n\}}
\quad\hbox{and}\quad
m\cdot \lceil t \rceil_{\{1,\ldots,m\}}\ge \lceil t \rceil_{\{m+1,\ldots,m+n\}}.
$$
Hence,
$$
\lceil t\rceil^\prime\ge \max(m,n)^{-1}\cdot \lceil t\rceil.
$$

By varying the sets $I_s$ over all possible two element admissible sets,
we see that {\sc Case $2_r$} can be equivalently formulated as: 
\[
\Delta(t_1,\ldots,t_r) > c_r \min\big(\lceil t_1 \rceil^\prime,\ldots  \lceil t_r \rceil^\prime\big).
\]
Then
\[
\Delta(t_1,\ldots,t_r) > c_r\max(m,n)^{-1} \min\big(\lceil t_1 \rceil,\ldots  \lceil t_r \rceil\big).
\]
This means that there exists an index $s$ such that 
$$
\Delta(t_1,\ldots,t_r) > c_r\max(m,n)^{-1} \lceil t_s \rceil.
$$
Without loss of generality, we may assume that 
\begin{equation}\label{eq:e0}
\Delta(t_1,\ldots,t_r) > c_r\max(m,n)^{-1} \lceil t_r \rceil.
\end{equation}
In this case, we introduce a function
$$
f_{r-1}:=\varphi_r\circ a(t_r).
$$
From \eqref{eq:e0},
$$
\big\|f_{r-1}|_Y\big\|_W\le \|f_{r-1}\|_{C^1}\ll e^{2\lceil t_r \rceil}\|\varphi_r\|_{C^1} \ll e^{2\max(m,n)/c_r\, \Delta(t_1,\ldots,t_r)} \|\varphi_r\|_{C^1},
$$
and 
$$
\int_Y \left(\prod_{s=1}^{r} \varphi_s \circ a(t_s)\right) \,  d\nu
=\int_Y f_{r-1}\left(\prod_{s=1}^{r-1} \varphi_s \circ a(t_s)\right) \,  d\nu.
$$
Now we apply the above reasoning to the later integral 
and consider the corresponding cases:

\vspace{0.2cm}
\noindent{\sc Case $1_{r-1}$:} {\it  
	we have
	$$
	\Delta(t_1,\ldots,t_r) > c_r \max(m,n)^{-1}\lceil t_r \rceil,
	$$
	and there are admissible subsets $I_1,\ldots,I_{r-1} \subset \{1,\ldots,m+n\}$ such that
	\[
	\Delta(t_1,\ldots,t_r) \leq c_{r-1} \min\big(\lfloor{t_1}\rfloor_{I_1}, \ldots, \lfloor{t_{r-1}}\rfloor_{I_{r-1}}\big).
	\]
}

\vspace{0.1cm}

This case splits as above into two complementary subcases:
	
\vspace{0.2cm}	
	
\noindent{\sc Case $1^\prime_{r-1}$:} {\it  
we have 
	\[
	\Delta(t_1,\ldots,t_r)\geq \lambda c_{r-1} \max\big(\lceil t_1 \rceil_{I_1^c},\ldots, \lceil t_{r-1} \rceil_{I_{r-1}^c}\big).
	\]
}

\vspace{0.1cm}

As in {\sc Case} $1^{\prime}_r$ above, we apply Lemma \ref{l:induction} with $\alpha = c_{r-1}^{-1}$ and $\beta = (\lambda c_{r-1})^{-1} < c_{r-1}^{-1}$,  and we can conclude that
\begin{align*}
	&\int_Y f_{r-1}\left(\prod_{s=1}^{r-1} \varphi_s \circ a(t_s)\right) \,  d\nu \\
	=&  \left(\int_Y f_{r-1}\,d\nu\right)\prod_{s=1}^{r-1} \int_{Y} \varphi_s \circ a(t_s) \,  d\nu\\
	&\quad+O\left(
	e^{-(\delta/c_{r-1} - 2r\ell/\lambda c_{r-1})\Delta(t_1,\ldots,t_r)} \, \big \|f_{r-1}|_Y\big\|_W{\prod}_{s=1}^{r-1} \|\varphi_s\|_{C^\ell}\right)\\
	=& \prod_{s=1}^r \int_{Y} \varphi_s \circ a(t_s) \,  d\nu\\
	&\quad+O\left(
	e^{-(\delta/(2c_{r-1})-2\max(m,n)/c_r)\, \Delta(t_1,\ldots,t_r)} \, {\prod}_{s=1}^r \|\varphi_s\|_{C^\ell} \right).
\end{align*}

\vspace{0.2cm}

\noindent{\textsc{Case $1^{\prime\prime}_{r-1}$}:} {\it 
	at least one of the sets
	\[
	J_s: = \big\{ j \in I_s^c \,  : \,  t_{s,j} > (1/\lambda c_{r-1})\Delta(t_1,\ldots,t_r) \big\},\quad s=1\ldots, r-1,
	\]
	is non-empty.}

\vspace{0.1cm}

As in {\sc Case $1^{\prime\prime}_{r}$}, we define the new sets
\[
I'_{s} := I_s \sqcup J_s,\quad s=1,\ldots,r-1,
\]
and note that
\[
\Delta(t_1,\ldots,t_r) \leq \lambda c_{r-1} \min(\lfloor t_1 \rfloor_{I_1'},\ldots,\lfloor t_r \rfloor_{I_r'}).
\]
Hence, we can now restart {\sc Case} $1_{r-1}$ with $\lambda c_{r-1}$ instead of $c_{r-1}$ and $I_s'$
instead of $I_s$. We apply this argument recursively.
If eventually we arrive to {\sc Case} $1^\prime_{r-1}$
and obtain the bound: 
	\begin{align*}
		\int_Y \left(\prod_{s=1}^r \varphi_s \circ a(t_s)\right) \,  d\nu 
		=& \prod_{s=1}^r \int_{Y} \varphi_s \circ a(t_s) \,  d\nu\\
		&+O\left(
		e^{-\big(\delta/(2 \lambda^{r(m+n)} c_{r-1})-2\max(m,n)/c_r\big)\Delta(t_1,\ldots,t_r)} \, {\prod}_{s=1}^r \|\varphi_s\|_{C^\ell} \right).
	\end{align*}
This provides a favourable estimate when 
$$
\delta/(2 \lambda^{r(m+n)} c_{r-1})-2\max(m,n)/c_r>0.
$$

The remaining case is:
\vspace{0.2cm}

\noindent {\sc Case $2_{r-1}$:} {\it 
	we have
	$$
	\Delta(t_1,\ldots,t_r) > c_r \max(m,n)^{-1}\lceil t_r \rceil,
	$$
	and for {all} admissible $I_1,\ldots,  I_{r-1} \subset \{1,\ldots,m+n\}$, 
	\[
	\Delta(t_1,\ldots,t_r) > c_{r-1} \min\big(\lfloor t_1\rfloor_{I_1},\ldots  \lfloor t_{r-1}\rfloor_{I_{r-1}}\big).
	\]
}

\vspace{0.1cm}

As above, the last inequality implies that 
\[
\Delta(t_1,\ldots,t_r) > c_{r-1}\max(m,n)^{-1} \min\big(\lceil t_1 \rceil,\ldots  \lceil t_{r-1} \rceil\big).
\]
This means that there exists $1\le s\le  r-1$ such that 
\[
\Delta(t_1,\ldots,t_r) > c_{r-1}\max(m,n)^{-1} \lceil t_s \rceil.
\]
Without loss of generality, we may assume that 
$$
	\Delta(t_1,\ldots,t_r) > c_{r-1}\max(m,n)^{-1} \lceil t_{r-1} \rceil.
$$
Since $c_r\ge c_{r-1}$, we have 
$$
\Delta(t_1,\ldots,t_r) > c_{r-1}\max(m,n)^{-1} \max\big( \lceil t_{r-1} \rceil, \lceil t_{r} \rceil\big).
$$
In this situation, we introduce the function
$$
f_{r-2}:=\big(\varphi_{r-1}\circ a(t_{r-1})\big)\cdot \big(\varphi_r\circ a(t_{r})\big)
$$
and apply the argument as above to
the product $f_{r-2} \left(\prod_{s=1}^{r-2} \varphi_s \circ a(t_s)\right)$.
We treat this product by considering {\sc Cases $1_{r-2}$} and $2_{r-2}$
defined as above. Then we apply this reasoning again to deal with the remaining case and so on.

\vspace{0.2cm}

Generally, the outlined procedure will allow us to reduce the argument
to the situation when 
\begin{equation}\label{eq:e}
\Delta(t_1,\ldots,t_r) > c_{k+1}\max(m,n)^{-1} \max\big( \lceil t_{k+1} \rceil, \ldots,\lceil t_{r} \rceil\big).
\end{equation}
Then we introduce the function
$$
f_{k}:=\prod_{s=k+1}^r \varphi_s\circ a(t_{s}).
$$
It follows from \eqref{eq:e} that 
\begin{align*}
\big\| f_{k}|_Y\big\|_W & \le 
\prod_{s=k+1}^r \big\| \varphi_s\circ a(t_{s})|_Y\big\|_W
\le \prod_{s=k+1}^r \|\varphi_s\circ a(t_{s})\|_{C^1}\ll \prod_{s=k+1}^r e^{2\lceil t_{s} \rceil} \|\varphi_s\|_{C^1}\\
&\ll e^{2(r-k)\max(m,n)/c_{k+1}\, \Delta(t_1,\ldots,t_k)} \prod_{s=k+1}^r \|\varphi_s\|_{C^1}.
\end{align*}
and
$$
\int_Y \left(\prod_{s=1}^{r} \varphi_s \circ a(t_s)\right) \,  d\nu=
\int_Y f_{k}\left(\prod_{s=1}^{k} \varphi_s \circ a(t_s)\right) \,  d\nu.
$$
To analyze this integral, we consider the cases: 

\vspace{0.2cm}
\noindent{\sc Case $1_k$:} {\it  
	we have
	$$
	\Delta(t_1,\ldots,t_r) > c_{k+1}\max(m,n)^{-1}\,\max\big( \lceil t_{k+1} \rceil,\ldots \lceil t_r \rceil\big),
	$$
	and 
	there are admissible subsets $I_1,\ldots,I_{k} \subset \{1,\ldots,m+n\}$ such that
	\[
	\Delta(t_1,\ldots,t_r) \leq c_{k} \min\big(\lfloor{t_1}\rfloor_{I_1}, \ldots, \lfloor{t_{k}}\rfloor_{I_{k}}\big)
	\]
}
	
\vspace{0.1cm}	
	
{\sc Case $1_k$} splits into two subcases:	
	
\vspace{0.2cm}
\noindent{\sc Case $1^\prime_k$:} {\it  
	we have  
	\[
	\Delta(t_1,\ldots,t_r)\geq \lambda c_{k} \max\big(\lceil t_1 \rceil_{I_1^c},\ldots, \lceil t_{k} \rceil_{I_{k}^c}\big).
	\]
}

Applying Lemma \ref{l:induction} with $\alpha = c_{k}^{-1}$ and $\beta = (\lambda c_{k})^{-1} < c_{k}^{-1}$,  we conclude that
\begin{align*}
	&\int_Y f_{k}\left(\prod_{s=1}^{k} \varphi_s \circ a(t_s)\right) \,  d\nu \\
	=&  \left(\int_Y f_{k}\,d\nu\right)\prod_{s=1}^{k} \int_{Y} \varphi_s \circ a(t_s) \,  d\nu\\
	&\quad\quad+O\left(
	e^{-(\delta/c_{k} - 2r\ell/\lambda c_{k})\Delta(t_1,\ldots,t_r)} \, \big\|f_{k}|_Y\big\|_W {\prod}_{s=1}^{k} \|\varphi_s\|_{C^\ell}\right)\\
	=& \left(\int_Y f_{k}\,d\nu\right) {\prod}_{s=1}^k \int_{Y} \varphi_s \circ a(t_s) \,  d\nu\\
	&\quad\quad+O\left(
	e^{-\big(\delta/(2c_{k})-2(r-k) \max(m,n)/c_{k+1}\big)\, \Delta(t_1,\ldots,t_r)} \, {\prod}_{s=1}^k \|\varphi_s\|_{C^\ell} \right).
\end{align*}
Our inductive assumption \eqref{eq:induction}, gives that 
$$
\int_Y f_{k}\,d\nu=\prod_{s=k+1}^r \int_{Y} \varphi_s \circ a(t_s) \,  d\nu
+O\left(
e^{-\eta(r-k)\, \Delta(t_1,\ldots,t_r)} \, {\prod}_{s=k+1}^r \|\varphi_s\|_{C^\ell} \right)
$$
for some $\eta(r-k)>0$. 
Therefore, the above estimate gives a favourable bound provided that 
$$
\delta/(2c_{k})-2(r-k) \max(m,n)/c_{k+1}>0.
$$

\vspace{0.2cm}

\noindent{\textsc{Case $1^{\prime\prime}_{k}$}:} {\it 
	at least one of the sets
	\[
	J_s := \big\{ j \in I_s^c \,  : \,  t_{s,j} > (1/\lambda c_k)\Delta(t_1,\ldots,t_r) \big\},\quad s=1\ldots, k
	\]
	is non-empty.}

\vspace{0.1cm}

We introduce the new sets
\[
I'_{s} := I_s \sqcup J_s,\quad s=1,\ldots,k,
\]
and note that
\[
\Delta(t_1,\ldots,t_r) \leq \lambda c_{k} \min\big(\lfloor t_1 \rfloor_{I_1'},\ldots,\lfloor t_r \rfloor_{I_k'}\big).
\]
We can restart {\sc Case} $1_{k}$ with $\lambda c_{k}$ instead of $c_{k}$ and $I_s'$ instead of $I_s$. We apply this argument recursively.
Then eventually we arrive to {\sc Case} $1^\prime_{k}$ at some point
and we obtain the bound: 
\begin{align*}
	&\int_Y f_k \left(\prod_{s=k+1}^r \varphi_s \circ a(t_s)\right) \,  d\nu \\
	=& \left(\int_Y f_k\, d\nu\right)\prod_{s=1}^k \int_{Y} \varphi_s \circ a(t_s) \,  d\nu\\
	&\quad\quad  +O\left(
	e^{-(\delta/(2 \lambda^{r(m+n)} c_{k})-2(r-k)\max(m,n)/c_{k+1})\Delta(t_1,\ldots,t_r)} \, {\prod}_{s=1}^r \|\varphi_s\|_{C^\ell} \right),
\end{align*}
or otherwise we end up in {\sc Case $2_{k}$} below.
We note that this bound is favourable provided that 
\begin{equation}\label{eq:c2}
\delta/(2 \lambda^{r(m+n)} c_{k})-2(r-k)\max(m,n)/c_{k+1}>0.
\end{equation}

The remaining case is:

\vspace{0.2cm}

\noindent {\sc Case $2_{k}$:} {\it 
	we have
	$$
	\Delta(t_1,\ldots,t_r) > c_{k+1}\max(m,n)^{-1}\max\big(\lceil t_{k+1} \rceil,\ldots \lceil t_r \rceil\big),
	$$
	and for {all} admissible $I_1,\ldots,  I_{k} \subset \{1,\ldots,m+n\}$,  we have
	\[
	\Delta(t_1,\ldots,t_r) > c_{k} \min\big(\lfloor t_1\rfloor_{I_1},\ldots  \lfloor t_{k}\rfloor_{I_{k}}\big).
	\]
}

\vspace{0.1cm}

As above, the last inequality implies that 
\[
\Delta(t_1,\ldots,t_r) > c_{k}\max(m,n)^{-1} \min\big(\lceil t_1 \rceil,\ldots  \lceil t_{k} \rceil\big).
\]
This means that there exists $1\le s\le  k$ such that 
\[
\Delta(t_1,\ldots,t_r) > c_{k}\max(m,n)^{-1} \lceil t_s \rceil.
\]
Without loss of generality, we may assume that 
$$
\Delta(t_1,\ldots,t_r) > c_{k}\max(m,n)^{-1} \lceil t_{k} \rceil.
$$
Since $c_{k+1}\ge c_{k}$, we have 
$$
\Delta(t_1,\ldots,t_r) > c_{k}\max(m,n)^{-1} \max\big( \lceil t_{k} \rceil,\ldots \lceil t_{r} \rceil\big).
$$

We analyze this case by considering {\sc Cases} 
$1^\prime_{k-1}$, $1^{\prime\prime}_{k-1}$, $2_{k-1}$ and so on.
Proceeding in this way, we end up with the final case when
\begin{equation}\label{eq:imposs}
\Delta(t_1,\ldots,t_r) > c_{1}\max(m,n)^{-1} \max\big( \lceil t_{1} \rceil,\ldots \lceil t_{r} \rceil\big).
\end{equation}
We observe that for some $1\le s_1,s_2\le r$ and $1\le k\le m+n$,
$$
\Delta(t_1,\ldots,t_r)=\|t_{s_1}-t_{s_2}\|=|t_{s_1,k}-t_{s_2,k}|\le \max\big(\lceil t_{s_1} \rceil, \lceil t_{s_2} \rceil\big).
$$
Since $c_1>\max(m,n)$, it follows that \eqref{eq:imposs} is impossible.
Hence, we have covered all the cases.

It remains to justify the choice of parameters $c_1,\ldots,c_r$. We recall that they must satisfy \eqref{eq:c1} and  \eqref{eq:c2}.
Therefore, we can choose them recursively so that $c_1>\max(m,n)$ and
$$
c_{k+1}>\max\left(1, 4(r-k)\max(m,n)\lambda^{r(m+n)}/\delta \right) c_k.
$$
This completes the proof of Theorem \ref{th:new}.

\end{document}